\documentclass[11pt,reqno]{amsart}
\usepackage{graphicx}
\usepackage{color}
\usepackage{amsmath,amssymb,amsthm,amsfonts}
\usepackage{graphicx}
\usepackage{color}
\usepackage{multicol}

\usepackage[numbers,sort&compress]{natbib}
\def\R{\mathbb{R}}
\def\a{\chi}
\def\b{\beta}
\def\k{\chi}
\def\ii{\int_0^\infty}
\def\be{\begin{equation}}
\def\ee{\end{equation}}

\newtheorem{lemma}{\bf Lemma}[section]
\newtheorem{theorem}{\bf Theorem}[section]
\newtheorem{proposition}{\bf Proposition}[section]

\newtheorem{remark}{\bf Remark}[section]

\numberwithin{equation}{section}
\newcommand{\abs}[1]{\left\vert#1\right\vert}

\begin{document}
\author{Jose A Carrillo}
\address{Department of Mathematics\\ Imperial College London, London SW7 2AZ, United Kingdom}
\email{carrillo@imperial.ac.uk}

\author{Jingyu Li}
\address{School of Mathematics and Statistics, Northeast Normal University, Changchun 130024, P.R. China}
\email{lijy645@nenu.edu.cn}

\author{Zhian Wang}
\address{Department of Applied Mathematics, Hong Kong Polytechnic University, Hung Hom, Kowloon, Hong Kong}
\email{mawza@polyu.edu.hk}
\title[Existence and stability of boundary spike-layer steady state]{Boundary spike-layer solutions of the singular Keller-Segel system: existence and stability}

\begin{abstract}
We exploit the existence and nonlinear stability of boundary spike/layer solutions of the Keller-Segel system with logarithmic singular sensitivity in the half space, where the physical zero-flux and Dirichlet boundary conditions are prescribed. We first prove that, under above boundary conditions, the Keller-Segel system admits a unique boundary spike-layer steady state where the first solution component (bacterial density) of the system concentrates at the boundary as a Dirac mass and the second solution component (chemical concentration) forms a boundary layer profile near the boundary as the chemical diffusion coefficient tends to zero. Then we show that this boundary spike-layer steady state is asymptotically nonlinearly stable under appropriate perturbations. As far as we know, this is the first result obtained on the global well-posedness of the singular Keller-Segel system with nonlinear consumption rate. We introduce a novel strategy of relegating the singularity, via a Cole-Hopf type transformation, to a nonlinear nonlocality which is resolved by the technique of ``taking antiderivatives'', i.e. working at the level of the distribution function. Then, we carefully choose weight functions to prove our main results by suitable weighted energy estimates with Hardy's inequality that fully captures the dissipative structure of the system.

\vspace{0.2cm}

\vspace{0.2cm}

\vspace{0.2cm}
\end{abstract}
\maketitle

%
%

\section{Introduction}
In their seminal work \cite{KS}, Keller and Segel proposed the following singular chemotaxis system
\begin{eqnarray} \label{KS}
\left\{
\begin{array}{lll}
u_{t}=u_{xx}-\chi[u(\ln w)_{x}]_{x}, \\[1mm]
w_{t}=\varepsilon w_{xx}-uw^{m},
\end{array}
\right.
\end{eqnarray}
to describe the propagation of traveling bands of chemotactic bacteria observed in the celebrated experiment of Adler \cite{Adler66}, where $u(x,t)$ denotes the bacterial density and $w(x,t)$ the oxygen/nutrient concentration. $\varepsilon\geq 0$ is the chemical diffusion coefficient, $\chi> 0$ denotes the chemotactic coefficient and $m\geq0$ the oxygen consumption rate. The system \eqref{KS} has been well-known as the singular Keller-Segel model nowadays as a cornerstone for the modeling of chemotactic movement in chasing nutrient.

The prominent feature of the Keller-Segel system \eqref{KS} is the use of a logarithmic sensitivity function $\ln w$, which was experimentally verified later in \cite{Kalinin}. This logarithm results in a mathematically unfavorable singularity which, however, has been proved to be necessary to generate traveling wave solutions (cf. \cite{LuiWang}) that were the first kind results obtained for the Keller-Segel system \eqref{KS}.  When $0\leq m<1$, Keller and Segel \cite{KS} have shown that the model (\ref{KS}) with $\varepsilon=0$ can generate traveling bands qualitatively in agreement with the experiment findings of \cite{Adler66}, and later the existence results of traveling wave solutions were extended to any $\varepsilon\geq 0$ and $0\leq m\leq 1$ (cf. \cite{LuiWang, NagaiIkeda, Odell75, Hartmut}), where the wave profile of $(u,w)$ is of (pulse, front) for $0\leq m<1$ and of (front, front) for $m=1$. {When $m>1$, it was proved that the system \eqref{KS} did not admit any type of traveling wave solutions (e.g., see \cite{Hartmut, Wang-TWS-DCDSBrev}). Though the Keller-Segel model \eqref{KS} with $m=1$ can not reproduce the pulsating wave profile to interpret the experiment of \cite{Adler66}, it was later employed to describe the boundary movement of bacterial chemotaxis \cite{Nossal72} and migration of endothelial cells toward the signaling molecule vascular endothelial growth factor (VEGF) during the initiation of angiogenesis (cf. \cite{LSN}).

Aside from the existence of traveling wave solutions, the logarithmic singularity become a source of difficulty in studying the Keller-Segel system \eqref{KS}, such as stability of traveling waves, global well-posedness and so on. When $m=1$, a Cole-Hopf type transformation was cleverly used to remove the singularity, which consequently led to a lot of interesting analytical works, for instance the stability of traveling waves  (cf. \cite{Davis, JLW, LW09, LW10, LW11, LW12, LLW, Chae, Choi2}),   global well-posedness and/or asymptotic behavior of solutions (see \cite{Choi1, GXZZ,Li-Pan-Zhao,peng-ruan-zhu2012global,MWZ,zhang-zhu2007,LPZ,LZ,Wang-Zhao13,TWW} in one dimensional bounded or unbounded space and \cite{LLZ,Hao,DL,PWZ,Rebholz,WXY,LPZ,WWZ} in multidimensional spaces) and boundary layer solutions \cite{HWZ,HLWW,HWJMPA}. However as far as we know no results have been available for the case $m\ne 1$ except the existence of traveling wave solutions as mentioned above. The main issue is that the Cole-Hopf type transformation used to resolving the logarithmic singularity worked effectively for the case $m=1$, but generated new analytical barriers hard to handle.  The purpose of this paper is to develop a novel strategy to break down these barriers and make some progress on the global dynamics (global existence and large-time behavior of solutions) of the singular Keller-Segel system \eqref{KS} for any $m\geq 0$.

We shall consider the Keller-Segel system \eqref{KS} in the half-space $\R_+=[0, \infty)$ with the following initial value
\begin{equation}\label{initial}
(u,w)(x,0)=(u_{0}(x),w_{0}(x)),\ \ x\in\R_+,\end{equation}
and boundary conditions
\begin{equation}\label{bou-cond}
\begin{cases}
(u_{x}-\chi u(\ln w)_{x})(0,t)=0, \ w(0,t)=b, \\
(u,w)(+\infty,t)=(0,0),
\end{cases}
\end{equation}
where $b>0$ is a constant denoting the boundary value of $w(x,t)$. That is we prescribe the zero-flux boundary conditions for $u$ and non-homogeneous Dirichlet boundary condition for $w$. Indeed such boundary conditions as \eqref{bou-cond} have been used in the chemotaxis-fluid model to reproduce the boundary accumulation layers formed by aerobic bacteria in the experiment of \cite{Tuval}. They are also consistent with the experimental conditions of Adler \cite{Adler66} where the nutrient was placed at one end of capillary tube.  It is worthwhile to note that boundary conditions \eqref{bou-cond} are different from Neumann boundary conditions that were often used in the literature for chemotaxis models. Hence no empirical results/methods are directly available for our concerned problem. Indeed with non-homogeneous Dirichlet boundary condition on $w$, the basic $L^2$-estimate becomes elusive in contrast to Neumann boundary conditions. In this paper, we shall develop some new ideas to  establish the existence, uniqueness and stability of steady states to the Keller-Segel system \eqref{KS}-\eqref{bou-cond} with $m\geq 0$. Specifically we show that
\begin{enumerate}
\item[(i)] The problem \eqref{KS}-\eqref{bou-cond} admits a unique non-constant steady state $(U,W)$, where $U$ forms a Dirac mass at the boundary $x=0$ as $\chi \to \infty$ or $\varepsilon \to 0$ and $W$ forms a boundary-layer profile as $\varepsilon \to 0$ (see Theorem \ref{prop2}).
\item[(ii)] The unique boundary spike/-layer steady state $(U,W)$ obtained above is asymptotically stable. Actually, we show that if the initial value $(u_0, w_0)$ is a small perturbation of the steady state $(U,W)$ in some topological sense, then the solution of \eqref{KS}-\eqref{bou-cond} will converge to $(U,W)$ point-wisely as time tends to infinity (see Theorem \ref{thm-1}).
\end{enumerate}
Resorting to the special structure of \eqref{KS} under the boundary conditions \eqref{bou-cond}, we are able to find the explicit steady state solution $(U,W)$ whose asymptotic profile as $\chi \to \infty$ or $\varepsilon \to 0$ can be determined. Therefore the result (i) above can be obtained without too much analytical effort. However, when proving the asymptotic stability of $(U,W)$ stated in (ii), we have to deal with the challenge of the logarithmic singularity. Our new idea of settling this difficulty is to transform the singular Keller-Segel system into a system with a nonlinear nonlocal term via a  Cole-Hopf type transformation (simply speaking we relegate the singularity to a nonlocality).   By fully exploiting the system structure and employing the ``technique of taking antiderivatives'', we convert this nonlocality into an exponential nonlinearity and then prove our desired results via the method of weighted energy estimates by carefully choosing weight functions. As far as we know, the results and ideas described above are new, and we achieve an understanding of the long time asymptotics for the singular Keller-Segel system \eqref{KS} with $m\ne 1$.

Though we consider the singular Keller-Segel system \eqref{KS} with any $m\geq 0$ in one dimension, the ideas developed in this paper may be applicable to multi-dimensional spaces. However, one has to face new difficulties. On one hand, the steady state $(U,W)$ can not be explicitly expressed in multi-dimensions, and on the other hand, the technique of ``taking antiderivatives'' ought to be associated with gradient and/or divergence operators. Moreover, the procedure of carrying out weighted energy estimates with appropriate weight functions will be sophisticated.

The paper is organized as follows. In section \ref{Sect.2}, we shall
derive the explicit formula of spiky-layer steady states, and state the main result of this paper
 on the asymptotic stability of spiky-layer steady states. Section \ref{sec.3}
is devoted to the proofs of our main results.

\section{Boundary spike/layer steady states}\label{Sect.2}
In this section, we first study the steady state problem of system
\eqref{KS}. The steady state can be solved
explicitly, and behaves like a (spike, layer) profile as $\varepsilon$ is small. We then present some
elementary calculations and state our main results on the asymptotic
stability of the spike.

With the zero-flux boundary condition on $u$, we immediately find that the bacterial mass is conserved, namely
\begin{eqnarray} \label{constrain}
\lambda:=\int_0^\infty u(x,t)dx=\int_0^\infty u_0(x)dx
\end{eqnarray}
which can be obtained directly by integrating the first equation of \eqref{KS} over $\R_+$. Therefore hereafter $\lambda>0$ is a prescribed number denoting cell mass.

The steady state of \eqref{KS} satisfying boundary condition \eqref{bou-cond} satisfies
\begin{eqnarray} \label{diff-equ}
\left\{
\begin{array}{lll}
U_{xx}-\chi (U(\ln  W)_x)_x=0,\\[1mm]
\varepsilon  W_{xx}-U W^m=0,\\[1mm]
\int_0^\infty U(x)dx=\lambda>0,
\end{array}
\right.
\end{eqnarray}
with boundary conditions
\begin{eqnarray} \label{boundary}
(U_x-\chi U(\ln  W)_x)(0)=0, \  W(0)=b, \ (U, W)(+\infty)=(0,0).
\end{eqnarray}

We first solve \eqref{diff-equ}-\eqref{boundary} explicitly.
\begin{proposition}\label{prop1}
Assume that $m\geq0$ and $\chi >|1-m|$. Then system
\eqref{diff-equ}-\eqref{boundary} has a unique solution $(U, W)$
satisfying $U'(x)<0$, $ W'(x)<0$, and
\begin{equation}\label{U}
U(x)=\frac{\lambda^{2}(\chi+1-m)^{2}}{2\varepsilon(\chi+m+1)b^{1-m}}
\left(1+\frac{\lambda(\chi+m-1)(\chi+1-m)}{2\varepsilon(\chi+m+1)b^{1-m}}\,
x\right)^{\frac{-2\chi}{\chi+m-1}},\end{equation}
\begin{equation}\label{C}
 W(x)=b\left(1+\frac{\lambda(\chi+m-1)(\chi+1-m)}{2\varepsilon(\chi+m+1)b^{1-m}}\, x\right)^{\frac{-2}{\chi+m-1}}.\end{equation}
\end{proposition}

\begin{proof}
The first equation of \eqref{diff-equ} and the boundary condition
\eqref{boundary} at $x=0$ give
$$ U_{x}=\chi U(\ln  W)_{x}.$$
Then there is a constant $c_{0}>0$ such that
\begin{eqnarray} \label{new1}
U(x)=c_{0} W^{\chi}.
\end{eqnarray}
Substituting \eqref{new1} into the second equation of
\eqref{diff-equ} leads to
$$\varepsilon  W_{xx}=c_{0} W^{\chi+m}.$$
Owing to the second equation of \eqref{diff-equ}, $ W_{xx}\geq0$, and
noting $ W_x(+\infty)=0$, we get
$$ W_x(x)\leq0, \text{ for } x\in[0,\infty).$$
Multiplying this equation by $ W_{x}$, and using the boundary
condition \eqref{boundary} at $x=+\infty$, we have
\begin{equation} \label{new2}
\frac{\varepsilon  W_{x}^{2}}{2}=\frac{c_{0}
W^{\chi+m+1}}{\chi+m+1}.
\end{equation}
It then follows from \eqref{new2} that
$$ W_{x}=-\left(\frac{2c_{0}}{\varepsilon(\chi+m+1)}\right)^{\frac{1}{2}}\, W^{\frac{\chi +m+1}{2}}.$$
For convenience, we denote
$$A:=\left(\frac{2}{\varepsilon(\chi+m+1)}\right)^{\frac{1}{2}},\ r:=\frac{\chi+m-1}{2}>0.$$
Then
\[\frac{1}{r}( W^{-r})_x=Ac_0^{\frac{1}{2}}.\]
This directly yields from \eqref{boundary} that
\begin{eqnarray} \label{newc}
 W(x)=\left(b^{-r}+rAc_{0}^{\frac{1}{2}}\, x\right)^{-\frac{1}{r}}.
\end{eqnarray}
We next determine the value of $c_{0}$. By \eqref{new1} and the
third equation of \eqref{diff-equ}, we have
$$c_{0}\int_0^\infty \left(b^{-r}+rAc_{0}^{\frac{1}{2}}\, x\right)^{-\frac{\chi}{r}}dx=\lambda.$$
Note that $-\frac{\chi}{r}+1=-\frac{\chi+1-m}{\chi+m-1}<0$ (due to $\chi>|1-m|$) gives
$-\frac{\chi}{r}<-1$. Then a simple computation yields
$$c_{0}=\frac{\lambda^{2}(\chi+1-m)^{2}}{2\varepsilon(\chi+m+1)b^{\chi+1-m}}.$$
Now substituting $c_{0}$ into \eqref{newc} and \eqref{new1}, we get
\eqref{U} and \eqref{C}, and thus finish the proof.
\end{proof}

Next we derive the asymptotic profile of the unique steady state $(U,W)$ given by formulas \eqref{U} and \eqref{C}, which turns out that the bacteria density $U$ forms a boundary spike as $\chi \to \infty$ or $\varepsilon \to 0$ and $W$ forms a boundary layer as $\varepsilon \to 0$.
\begin{theorem}\label{prop2}
Let $m\geq 0$ and $\chi>|1-m|$ and $(U, W)$ be the unique solution of \eqref{diff-equ}-\eqref{boundary} obtained in Proposition
\ref{prop1}. Then the following results hold.
\begin{enumerate}
\item[(i)] As $\chi\rightarrow\infty$, $U$
concentrates at $x=0$ and $W$ converges to the boundary value $b$ on any bounded interval. That is
\[\begin{split}
&U(x)\rightarrow\lambda\delta(x) \text{ in the sense of distribution as} \ \chi\rightarrow\infty,
\\& W(x)\rightarrow b \text{ uniformly in} \ [0,N] \ \text{for any} \ 0<N<\infty \ \text{as}\ \chi\rightarrow\infty.
\end{split}\]
\item[(ii)] As $\varepsilon \to 0$, $U$  concentrates at $x=0$ and $W(x)$ forms a (boundary) layer near $x=0$. Namely
$$
U(x)\rightarrow\lambda\delta(x) \text{ in the sense of distribution as}\ \varepsilon \to 0
$$
and there is a constant $\eta=\eta(\varepsilon)$ satisfying $\varepsilon/\eta(\varepsilon)\to 0$ as $\varepsilon \to 0$ such that
$$
\lim\limits_{\varepsilon\rightarrow0}\| W\|_{L^\infty[\eta, \infty]}=0, \ \ \ \liminf\limits_{\varepsilon\rightarrow 0} \|
W\|_{L^\infty[0,\infty)}>0.
$$
\end{enumerate}
\end{theorem}
\begin{proof}
We first prove (i). For any $\zeta(x)\in C_0^\infty[0,\infty)$ and any $h>0$, we have
\begin{equation}\label{2.9}
\begin{split}\ii U(x)\zeta(x)dx-\lambda\zeta(0)&=\ii U(x)(\zeta(x)-\zeta(0))dx\\
&=\int_0^h U(x)(\zeta(x)-\zeta(0))dx+\int_h^\infty
U(x)(\zeta(x)-\zeta(0))dx.\end{split}\end{equation}
On  one hand, for any $x>0$ we can rewrite \eqref{U} as
\begin{equation*}
\begin{split}
U(x)=\frac{\lambda^{2}}{2\varepsilon
b^{1-m}}\,\frac{(1+\frac{1-m}{\chi})^{2}}{1+\frac{m+1}{\chi}}\,\left(\frac{1}{\chi
x} +\frac{\lambda}{2\varepsilon
b^{1-m}}\,\frac{1-\frac{(1-m)^2}{\chi^2}}{1+\frac{m+1}{\chi}}\right)^{\frac{-2\chi}{\chi+m-1}}\, \frac{\chi^{-1-\frac{2(1-m)}{\chi+m-1}}}{
x^{\frac{2\chi}{\chi+m-1}}}.\end{split}
\end{equation*}
It is easy to see that $U(x)\rightarrow0$ uniformly on $[h,\infty)$
as $\chi\rightarrow\infty$. It then follows from Lebesgue
Dominated Convergence Theorem that
\[\int_h^\infty U(x)(\zeta(x)-\zeta(0))dx\rightarrow0 \text{ as } \chi\rightarrow\infty.\]
On the other hand, since $\zeta(x)\in C_0^\infty[0,\infty)$, there is a constant $C_0$ such that $|\zeta(x)-\zeta(0)|=|\zeta'(\theta)|x\leq C_0x$. Thus it follows that
\[\left|\int_0^h U(x)(\zeta(x)-\zeta(0))dx\right|\leq C_0\int_0^h xU(x)dx \leq C_0\lambda h.\]
It hence follows from \eqref{2.9}  that
\[\underset{\chi\rightarrow\infty}{\overline{\lim}}\left|\ii U(x)\zeta(x)dx-\lambda\zeta(0)\right|\leq C_0\lambda h, \ \forall h>0,\]
which implies
\[U(x)\rightarrow\lambda\delta(x) \text{ as }\chi\rightarrow\infty.\]
To derive the limit of $ W(x)$, we note that there exists a constant
$C_1>0$, such that for all $x\in[0,N]$ and large $\chi$, it holds that
\[\begin{split}
1&\leq\left(1+\frac{\lambda(\chi+m+1)(\chi+1-m)}{2\varepsilon(\chi+m+1)b^{1-m}}\,
x\right)^{\frac{1}{\chi+m-1}}=\left(1+\frac{\lambda}{2\varepsilon
b^{1-m}}\,\frac{1-\frac{(1-m)^2}{\chi^2}}{1+\frac{m+1}{\chi}}\,
\chi
x\right)^{\frac{1}{\chi+m-1}}\\&\leq(1+C_1N\chi)^{\frac{1}{\chi+m-1}} \rightarrow1
\text{ as }\chi\rightarrow\infty.\end{split}\] This implies $
W(x)\rightarrow b$ uniformly on any bounded interval.

Next we prove (ii). To this end,  we rewrite $(U,W)(x)$ as
\begin{equation}\label{UWnew}
U(x)=\frac{\theta \sigma \xi}{2\chi \varepsilon} \Big(1+\frac{\sigma}{\varepsilon}x\Big)^{-\xi}, \ W(x)=b\Big(1+\frac{\sigma}{\varepsilon}x\Big)^{-\xi/\chi}
\end{equation}
with $\theta=\lambda(\chi+1-m)>0, \ \sigma=\frac{\lambda(\chi+1-m)(\chi+m-1)}{2(\chi+m+1)b^{1-m}}>0, \ \xi=\frac{2\chi}{\chi+m-1}$.
Note that $\xi>1$ since $\chi>|1-m|$. Then one can verify that $U(x)\rightarrow0$ uniformly on $[h,\infty)$
as $\varepsilon\rightarrow\infty$ for $h>0$. By the same argument as proving case (i), we have that $U(x)\rightarrow\lambda\delta(x) \text{ in the sense of distribution as}\ \varepsilon \to 0$. Now we proceed to prove $W(x)$ forms a boundary layer near $x=0$. Indeed it can be directly checked from \eqref{UWnew} that for $\eta(\varepsilon)=O(\varepsilon^\alpha)$ with $0<\alpha <1$, $W(x) \to 0$ uniformly on $[\eta(\varepsilon), \infty)$ as $\varepsilon \to 0$ (namely $\lim\limits_{\varepsilon\rightarrow0}\|W\|_{L^\infty[\eta(\varepsilon), \infty]}=0$).
On the other hand, it is obvious that $\liminf\limits_{\varepsilon\rightarrow 0} \|W\|_{L^\infty[0,\infty)}=b>0$. This implies $W(x)$ develops a boundary layer on $[0,\eta(\varepsilon)]$ as $\varepsilon \to 0$ and hence completes the proof.
\end{proof}
To illustrate our results, we numerically plot the asymptotic profiles of $(U,W)$ in Fig.\ref{fig1} for $\chi \to \infty$ and in Fig.\ref{fig2} for $\varepsilon \to 0$.  From Fig.\ref{fig1}, we see that the value of $U(0)$ increases as $\chi$ increases and $U$ behaves like a spike (Dirac delta function) concentrating at the boundary $x=0$, while $W$ is elevated towards the boundary value $b=1$ as $\chi$ increases. This verifies the results of Theorem \ref{prop2}(i). Fig.\ref{fig2} demonstrates the asymptotic profile of $U$ and $W$ as $\varepsilon$ decreases to zero, where we observe that $U$ tends to aggregate at the boundary $x=0$ like a Dirac delta function while $W$ tends to vanish in the interior of the domain (outer-layer region) but remains positive  in the region close to the boundary $x=0$ (inner-layer region) as $\varepsilon$ decreases. In particular, the slope of curve $W$ becomes increasingly steeper at $x=0$ as $\varepsilon$ decreases. This implies that $W(x)$ develops a boundary layer profile as $\varepsilon$ is small, which is well consistent with the results of Theorem \ref{prop2}(ii).

\begin{figure}[htbp]
\centering
\includegraphics[width=7.5cm,angle=0]{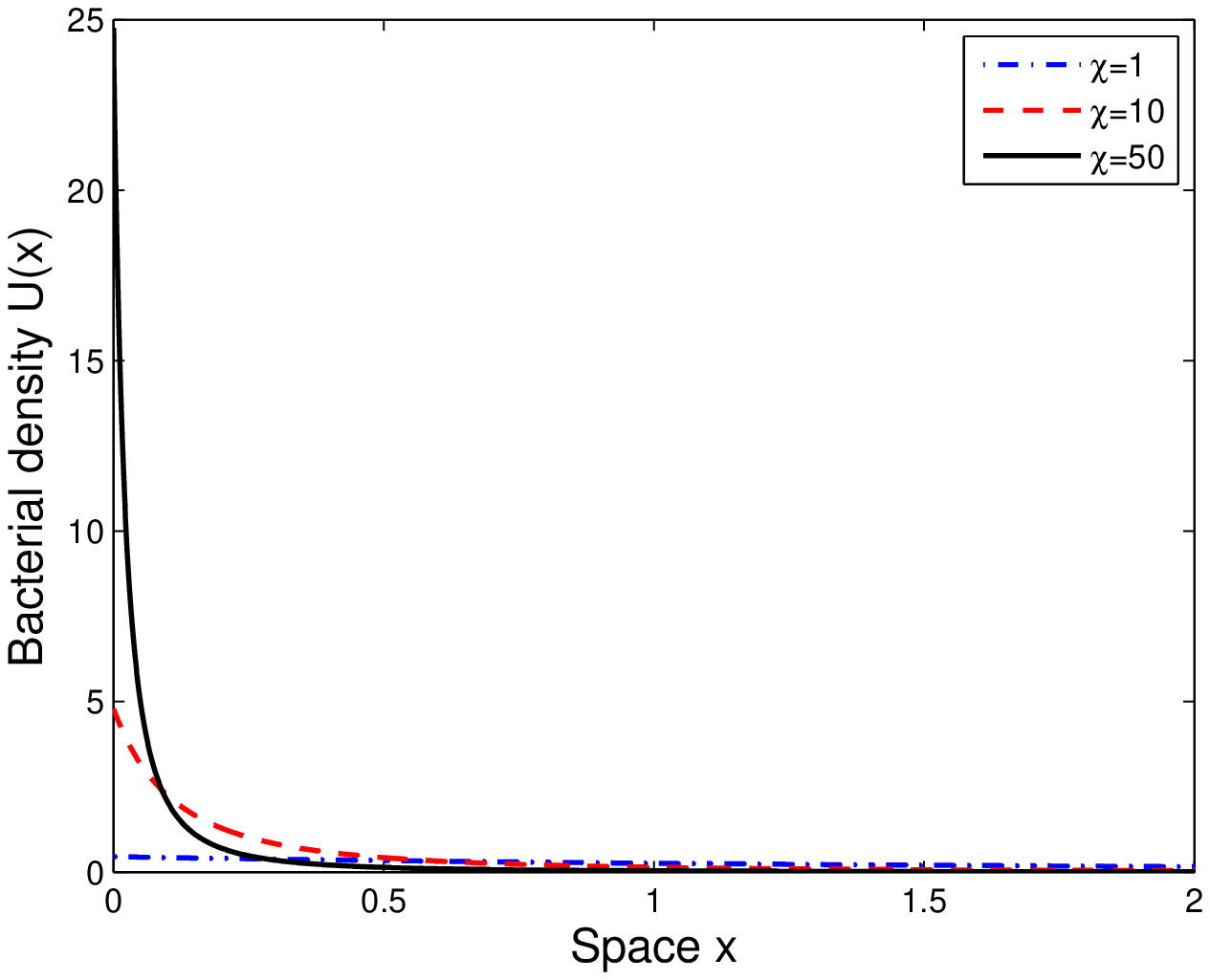}
\includegraphics[width=7.5cm,angle=0]{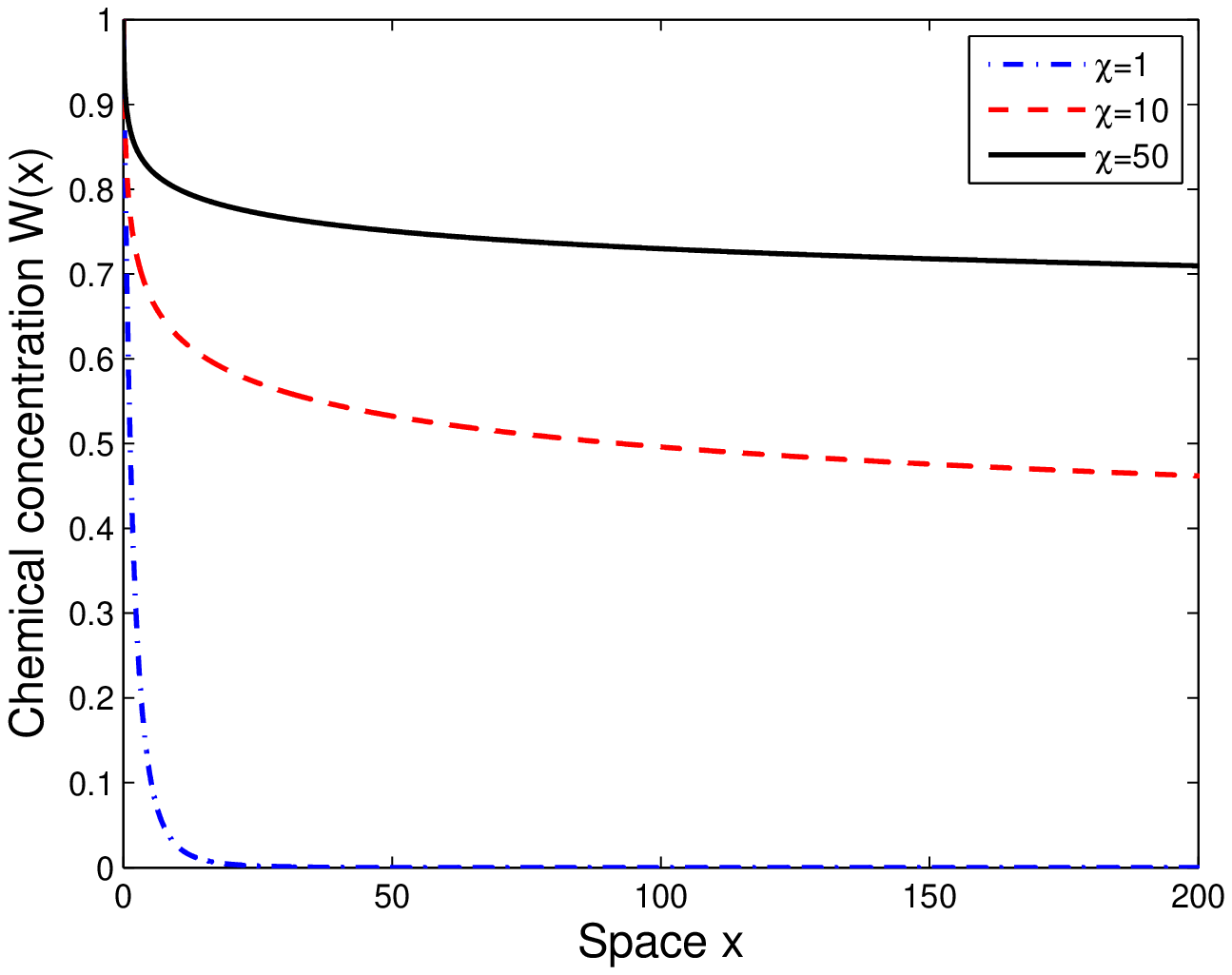}
\caption{Profiles of steady state $(U,W)(x)$ with $b=\lambda=\varepsilon=1$, $m=0.5$ for different values of $\chi>0$.}
\label{fig1}
\end{figure}

\begin{figure}[htbp]
\centering
\includegraphics[width=7.5cm,angle=0]{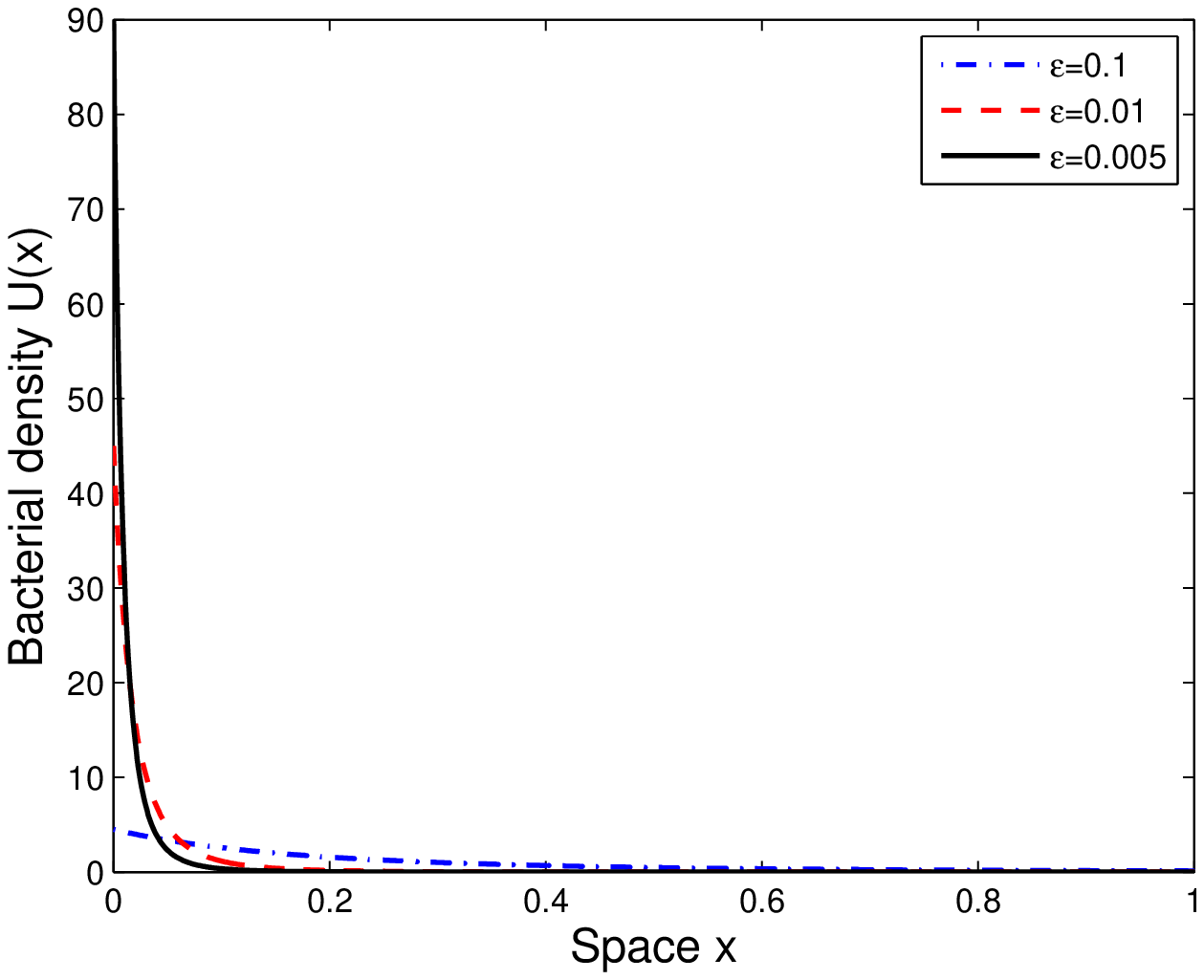}
\includegraphics[width=7.5cm,angle=0]{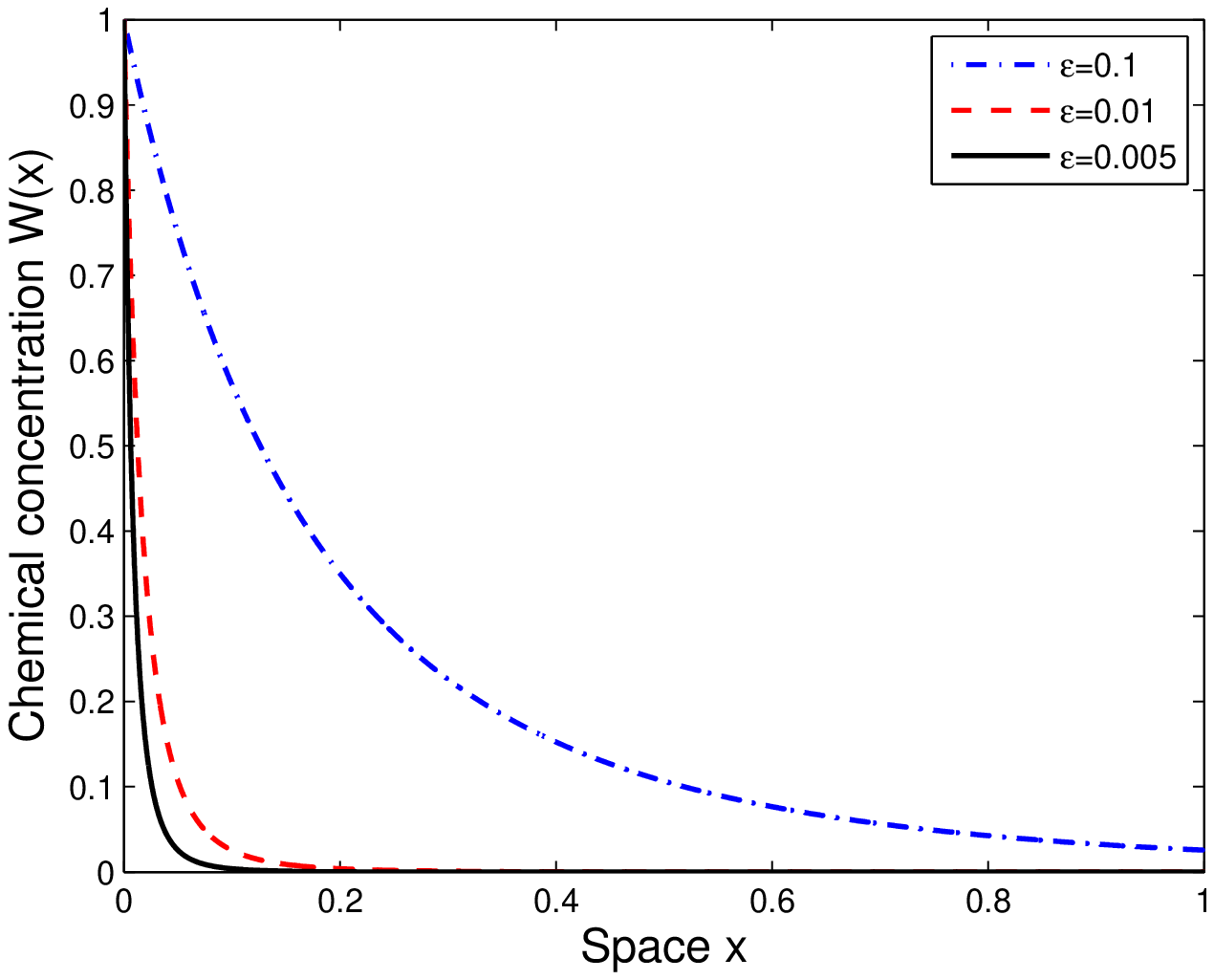}
\caption{Asymptotic profiles of steady state $(U,W)(x)$ with respect to $\varepsilon>0$, where $b=\lambda=\chi=1$, $m=0.5$.}
\label{fig2}
\end{figure}

We next study the asymptotic stability of the steady state $(U,W)$ to the system
\eqref{KS}-\eqref{constrain}. Because the chemical concentration
$w(x,t)$ has a vacuum end state, the first equation of Keller-Segel system \eqref{KS}
encounters a singularity at $x=\infty$ which makes a very difficult task to work with \eqref{KS} directly. To overcome such difficulty, we employ a Cole-Hopf type transformation
\begin{equation} \label{trans}
v:=-\frac{w_{x}}{w},\text { i.e. }(\ln w)_{x}=-v,
\end{equation}
which gives
\begin{equation}\label{eqc}
w(x,t)=be^{-\int_0^x v(y,t)dy}
\end{equation}
due to \eqref{trans} and boundary condition $w(0,t)=b$, and hence transforms system \eqref{KS} into a nonlocal parabolic-parabolic system of conservation laws as follows
\begin{eqnarray}\label{newuv}
\left\{
\begin{array}{lll}
u_{t}=u_{xx}+\chi(uv)_{x}, & \ (x,t)\in\R_+\times\R_+\\[1mm]
v_{t}=\varepsilon v_{xx}-(\varepsilon v^{2}-uw^{m-1})_{x}, & \ (x,t)\in\R_+\times\R_+\\[1mm]
w(x,t)=be^{-\int_0^x v(y,t)dy},\\[1mm]
(u,v)(x,0)=(u_{0}(x),v_{0}(x))
\end{array}
\right.
\end{eqnarray}
where $v_0=\frac{w_{0x}}{w}$. Before proceeding, we should remark that although the singularity is removed via the Cole-Hopf transformation \eqref{trans}, the price we pay is that the transformed system \eqref{newuv} has a nonlocal term and quadratic advection term which also bring tremendous difficulty to mathematical analysis.  However in the case $m=1$, the nonlocal term naturally vanishes and the system \eqref{newuv} becomes more tractable. There have been a large amount of results available to \eqref{newuv} with $m=1$ as recalled in the Introduction. We particulary remark that when Dirichlet boundary conditions are imposed to \eqref{newuv} with $m=1$, the existence and stability of boundary layer solutions have been shown recently in \cite{HWZ,HLWW,HWJMPA} where, however, the original Keller-Segel system \eqref{KS} was found to have no boundary layer solutions when reversing the results of \eqref{newuv} to $v$ via \eqref{trans}.  In this paper, we shall consider entirely different boundary conditions so that boundary spike and layer solutions can develop from the Keller-Segel system \eqref{KS} for any $m\geq 0$. When $m\ne 1$, then the second equation of \eqref{newuv} contains an advection including both quadratic nonlinearity and a nonlocal term, which leads to a very challenging problem.  As far as we know, there was not any result available for \eqref{newuv} with $m\ne 1$. In this paper, we shall develop some novel ideas to exploit the system \eqref{newuv} and hence obtain the first results on the original Keller-Segel model \eqref{KS} with $m\neq 1$ subject to the boundary condition \eqref{bou-cond} by studying the transformed nonlocal system  \eqref{newuv}. Next to state our main results, we derive the boundary conditions of $v(x,t)$. The
second equation of \eqref{KS} also gives
\begin{equation*}
\begin{split}
(\ln w)_{t}=\varepsilon\left(\frac{w_{x}}{w}\right)_{x}+\varepsilon\left(\frac{w_{x}}{w}\right)^{2}-uw^{m-1}
= -\varepsilon v_{x}+\varepsilon v^{2}-uw^{m-1}.
\end{split}
\end{equation*}
Because $b$ is a constant, for smooth solutions $(\ln w)_{t}=0\text{
at }x=0$, it then follows that
$$\varepsilon v_{x}-(\varepsilon v^{2}-uw^{m-1})=0\text{ at }x=0.$$
Denote by $(U,V)(x)$ the steady state of \eqref{newuv} where $U(x)$ is explicitly given in \eqref{U}.  Then by
\eqref{trans} and Proposition \ref{prop1}, we find $V$ given as
\begin{equation*}\label{V}
V(x)=-\frac{W_x}{W}=\frac{\lambda(\chi+1-m)}{\varepsilon
(\chi+m+1)b^{1-m}}\,\left(1+\frac{\lambda(\chi+m-1)(\chi+1-m)}{2\varepsilon
(\chi+m+1)b^{1-m}}\, x \right)^{-1}.
\end{equation*}
It can be easily verified that
\[V(x)\rightarrow0 \text{ as }x\rightarrow+\infty.\]
Since we are devoted to proving that $v(x,t) \to V(x)$ as $t \to \infty$, the following  condition is naturally imposed:
$v(+\infty,t)=0$,
which requires that $\frac{w_x}{w} \to 0$ as $x \to \infty$. Therefore the boundary conditions for \eqref{newuv} relevant to \eqref{bou-cond} is
\begin{eqnarray}\label{bound2}
\begin{cases}
u_{x}+\chi uv=\varepsilon v_{x}-(\varepsilon
v^{2}-uw^{m-1})=0, & x=0\\
 (u,v)\to (0,0), & x \to \infty.
\end{cases}
\end{eqnarray}
From Proposition \ref{prop1}, one can check that $(U,V)(x)$ is a unique steady state of \eqref{newuv}-\eqref{bound2}. In the following, we shall focus on attention to study the well-posedness and asymptotic behavior of solutions to the initial-boundary value problem \eqref{newuv}-\eqref{bound2} when the initial value $(u_0,v_0)$ is a small perturbation of $(U,V)(x)$.

Because the steady state has a vacuum end state which leads to a singularity in the energy estimates, as to be seen later, we have to study its
stability in carefully selected weighted functional spaces to resolve the singularity, where the weights
depend on the range of $m$. To state our results more precisely, we
denote by $H^k(k\geq0)$ the usual Sobolev space whose norm is
abbreviated as $\|f\|_k:=\sum\limits_{j=0}^{k}\|\partial_x^jf\|$
with $\|f\|:=\|f\|_{L^2(\mathbb{R}_+)}$, and $H^k_\omega$ denotes
the weighted Sobolev space of measurable function $f$ such that
$\sqrt{\omega}\partial_x^jf\in L^2(\mathbb{R}_+)$ for $0\leq j\leq
k$ with norm
$\|f\|_{k,\omega}:=\sum\limits_{j=0}^{k}\|\sqrt{\omega}\partial_x^jf\|$
and $\|f\|_\omega:=\|\sqrt{\omega}f\|_{L^2(\mathbb{R}_+)}$.\\

Our main results are stated as follows.

\begin{theorem}\label{thm-1}
Assume that $m\geq0$ and that $\chi>|1-m|$. Let $(U,V)$ be
the unique steady state of system \eqref{newuv}-\eqref{bound2}.
Assume that the initial perturbation around $(U,V)$ satisfies
$\phi_0(\infty)=\psi_0(\infty)=0$ where
\[(\phi_0,\psi_0)(x)=\int_0^x (u_{0}(y)-U(y),v_{0}(y)-V(y))dy.\]

\begin{enumerate}

\item If $m\geq1$, then there exists a constant $\delta_0>0$ such that if
$\|\phi_{0}\|^2_{1,w_1}+\|\psi_{0}\|^2_{1,w_2}+\|\phi_{0xx}\|^2+\|\psi_{0xx}\|^2\leq
\delta_0$, where $w_1=1/U$ and $w_2= W^{1-m}$, then the system
\eqref{newuv}-\eqref{bound2} has a unique global solution
$(u,v)(x,t)$ satisfying
\begin{equation}\label{regularity}
\begin{cases}
u-U\in C([0,\infty); H^1\cap L^2_{w_1})\cap L^2((0,\infty);H^2\cap H^1_{w_1}),\\
v-V\in C([0,\infty); H^1\cap L^2_{w_2})\cap L^2((0,\infty);H^2\cap
H^1_{w_2}).
\end{cases}\end{equation}

\item If $0\leq m<1$ and $\chi\gg1$, then there exists a constant $\delta_1>0$ such that if
$\|\phi_{0}\|^2_{1,w_3}+\|\phi_{0xx}\|^2+\|\psi_{0}\|^2_{2}\leq
\delta_1$, where $w_3=W^{m-1}/U$, then the system
\eqref{newuv}-\eqref{bound2} has a unique global solution
$(u,v)(x,t)$ satisfying
\begin{equation}\label{2.18}
\begin{cases}
u-U\in C([0,\infty); H^1\cap L^2_{w_3})\cap L^2((0,\infty);H^2\cap H^1_{w_3}),\\
v-V\in C([0,\infty); H^1)\cap L^2((0,\infty);H^2).
\end{cases}\end{equation}
\item In both cases (1) and (2) above, we have the following asymptotic convergence:
\begin{equation}\label{asym}
\sup\limits_{x\in\R_+}\abs{(u,v)(x,t)-(U,V)(x)}\to 0~~\text{as}~~
t\to+\infty,\\
\end{equation}
and
\begin{equation}\label{L1}
\|u(\cdot,t)-U(\cdot,t)\|_{L^1(\R_+)}\to 0 \ \text{as}  \ t\to+\infty.
\end{equation}
\end{enumerate}

\end{theorem}

By using the Cole-Hopf transformation \eqref{trans}, we transfer
Theorem \ref{thm-1} to the original Keller-Segel system
\eqref{KS}-\eqref{bou-cond}.

\begin{theorem}\label{thm2}
Assume that $m\geq0$ and that $\chi>|1-m|$. Let $(U,W)$ be
the unique steady state of  \eqref{KS}-\eqref{bou-cond}.
Assume that the initial perturbation  satisfies
$\phi_0(\infty)=\psi_0(\infty)=0$ where
\[\phi_0(x)=\int_0^x (u_{0}(y)-U(y))dy, \ \psi_0(x)=-\ln w_0(x)+\ln W(x).\]

\begin{enumerate}

\item If $m\geq1$, then there exists a constant $\delta_2>0$ such that if
$\|\phi_{0}\|^2_{1,w_1}+\|\psi_{0}\|^2_{1,w_2}+\|\phi_{0xx}\|^2+\|\psi_{0xx}\|^2\leq
\delta_0$, then the system \eqref{KS}-\eqref{bou-cond} has a unique
global solution $(u,w)(x,t)$ satisfying
\begin{equation*}\label{new-regularity}
\begin{cases}
u-U\in C([0,\infty); H^1\cap L^2_{w_1})\cap L^2((0,\infty);H^2\cap H^1_{w_1}),\\
w-W\in C([0,\infty); H^1)\cap L^2((0,\infty);H^2).
\end{cases}\end{equation*}

\item If $0\leq m<1$ and $\chi\gg1$, then there exists a constant $\delta_3>0$ such that if
$\|\phi_{0}\|^2_{1,w_3}+\|\phi_{0xx}\|^2+\|\psi_{0}\|^2_{2}\leq
\delta_3$, then the system \eqref{KS}-\eqref{bou-cond} has a unique
global solution $(u,w)(x,t)$ satisfying
\begin{equation*}
\begin{cases}
u-U\in C([0,\infty); H^1\cap L^2_{w_3})\cap L^2((0,\infty);H^2\cap H^1_{w_3}),\\
w-W\in C([0,\infty); H^1)\cap L^2((0,\infty);H^2).
\end{cases}\end{equation*}
\item In either of the above  cases (1) or (2), we have the following asymptotic convergence:
\begin{equation*}
\sup\limits_{x\in\R_+}\abs{(u,v)(x,t)-(U,V)(x)}\to 0~~\text{as}~~
t\to+\infty,\\
\end{equation*}
and
$$\|u(\cdot,t)-U(\cdot,t)\|_{L^1(\R_+)}\to 0 \ \text{as}  \ t\to+\infty.$$
\end{enumerate}
\end{theorem}

It is worthy to point out that in the previous theorems the $L^1$ convergence of the cell density is obtained as a consequence of the convergence in relative $L^2$-entropy, see its proof in section 3 for details.

\section{Stability of the spike/layer steady state (Proof of Theorem \ref{thm-1})}\label{sec.3}

In this section, we first prove Theorem \ref{thm-1} by using the
weighted energy method. We divide the proofs into two parts $m\geq1$
and $0\leq m<1$. In the latter case, the Hardy inequality plays an
important role to capture the full dissipative structures of the
system. Finally, we transfer the stability of $(U,V)$ for system
\eqref{newuv}-\eqref{bound2} back to the original Keller-Segel
system \eqref{KS}-\eqref{constrain}, and prove that the steady
state $(U,W)$ is asymptotically stable.
\subsection{Reformulation of the problem}
The steady state  $(U,V)$ of system \eqref{newuv}-\eqref{bound2}
satisfies
\begin{eqnarray} \label{ssuv}
\left\{
\begin{array}{lll}
U_{xx}+\chi(UV)_{x}=0,\\
\varepsilon V_{xx}-(\varepsilon V^{2}-UW^{m-1})_{x}=0,
\end{array}
\right.
\end{eqnarray}
with boundary conditions
\begin{equation*}
(U_{x}+\chi UV)(0)=(\varepsilon V_{x}-(\varepsilon
V^{2}-UW^{m-1}))(0)=0,\ (U,V)(+\infty)=(0,0).
\end{equation*}
Integrating \eqref{ssuv} in $x$ gives
\begin{eqnarray} \label{ssuv-2}
\left\{
\begin{array}{lll}
U_{x}+\chi(UV)=0,\\
\varepsilon V_{x}-(\varepsilon V^{2}-UW^{m-1})=0.
\end{array}
\right.
\end{eqnarray}

In view of \eqref{bound2}, $(u,v)$ actually satisfies the no-flux
boundary conditions. The perturbation around $(U,V)$ should have
the conservation of mass. In other words, it holds that
\begin{equation}\label{3.4}
\ii(u(x,t)-U(x),v(x,t)-V(x))dx=\ii(u_0(x)-U(x),v_0(x)-V(x))dx=(0,0).
\end{equation}
This fact stimulates us to employ the technique of anti-derivative
to study the asymptotic stability of steady state $(U,V)$. More importantly, we find that once we take the anti-derivative for $v$, the nonlocal term in $w$ (see \eqref{eqc}) will be removed. This key observation helps us find a potential way to deal with the nonlocal effect.   Therefore we
decompose the solution $(u,v)$ as
\begin{equation}\label{decomposition}
\phi_{x}=u-U,\ \psi_{x}=v-V.\end{equation}
Then
\[(\phi,\psi)(x,t)=\int_0^x (u(y,t)-U(y),v(y,t)-V(y))dy.\]
Substituting \eqref{decomposition} into \eqref{newuv}, integrating
the equations in $x$, and using \eqref{ssuv}, we get
\begin{eqnarray} \label{newfai}
\begin{cases}
\phi_{t}=&\phi_{xx}+\chi V\phi_{x}+\chi U\psi_{x}+\chi\phi_{x}\psi_{x},\ (x,t)\in\R_+\times\R_+,\\[1mm]
\psi_{t}=&\varepsilon\psi_{xx}-2\varepsilon V\psi_{x}-U
W^{m-1}(1-e^{-(m-1)\psi})+
W^{m-1}\phi_{x}\\
&-\varepsilon\psi_{x}^{2}-
W^{m-1}(1-e^{-(m-1)\psi})\phi_{x},
\end{cases}
\end{eqnarray}
where the initial value $(\phi,\psi)(x,0)$ is given by
\begin{equation}\label{newbound1}
 (\phi,\psi)(x,0)=(\phi_0,\psi_0)(x)=\int_0^x (u_{0}(y)-U(y),v_{0}(y)-V(y))dy,
\end{equation} which satisfies
\[(\phi_0,\psi_0)(+\infty)=(0,0)\]
and the boundary condition satisfies from \eqref{3.4} that
\begin{equation}\label{newbound2}
 (\phi,\psi)(0,t)=(0,0), \ (\phi,\psi)(+\infty,t)=(0,0),
\end{equation}

We remark that the second equation of \eqref{newfai} does not contain the term  $W^{m-1} \phi_x$ originally. Here we artificially add and subtract this term in the second equation of \eqref{newfai} in order to cancel the trouble ``cross'' terms in the energy estimates. This treatment is indeed a very important trick introduced in this paper. We finally comment that working at the level of the antiderivatives is in some sense related to ideas used in Keller-Segel models stemming from optimal transport as in \cite{BCC}. It turns out the analysis for the case $m\geq 1$ and $0\leq m< 1$ are quite different. Hence in the following we shall separate these two cases to discuss.

\subsection{Case $m\geq1$}\label{sec.3.1}

We look for solutions of system \eqref{newfai} with \eqref{newbound1}
and \eqref{newbound2} in the space
\begin{equation*}
\begin{split}
X(0,T):=\{&(\phi,\psi)(x,t)\big|\phi\in C([0,T]; H^2\cap
H^1_{w_1}),\phi_x\in L^2((0,T);H^2\cap H^1_{w_1}),\\&\psi\in C([0,T];
H^2\cap H^1_{w_2}),\psi_x\in L^2((0,T);H^2\cap H^1_{w_2})\},
\end{split}\end{equation*}
for $T\in(0,+\infty]$, where $w_1=1/U$ and $w_2=W^{1-m}$. Set
\begin{equation*}
N(t):=\sup_{\tau\in[0,t]}(\|\phi(\cdot,\tau)\|_{1,w_1}+\|\phi_{xx}(\cdot,\tau)\|+\|\psi(\cdot,\tau)\|_{1,w_2}
+\|\psi_{xx}(\cdot,\tau)\|).
\end{equation*}
Since
\begin{equation}\label{Un}
U(x)\leq\frac{\lambda^{2}(\chi+1-m)^{2}}{2\varepsilon(\chi+m+1)b^{1-m}}=: \bar{u}
 \end{equation}
 and $W(x)\leq b$ for $x\in \R_+$, we have
\begin{equation}\label{w}
w_1\geq \frac{2\varepsilon(\chi+m+1)b^{1-m}}{\lambda^{2}(\chi+1-m)^{2}}>0 \ \text{and}\ w_2\geq1/b^{1-m}>0
\end{equation}
since $m\geq1$. Thus the Sobolev embedding theorem implies
\begin{equation*}
\sup_{\tau\in[0,t]}\{\|\phi(\cdot,\tau)\|_{L^\infty},\|\phi_x(\cdot,\tau)\|_{L^\infty},\|\psi(\cdot,\tau)\|_{L^\infty},
\|\psi_x(\cdot,\tau)\|_{L^\infty}\}\leq N(t).
\end{equation*}

For  system \eqref{newfai}-\eqref{newbound2}, we have the following results.
\begin{proposition}\label{global-existence}
Assume that $m\geq1$ and
$\chi>|1-m|$. Then there exists a constant $\delta_1$, such
that if $N(0)\leq\delta_1$, the system \eqref{newfai}-\eqref{newbound2} has a unique global solution
$(\phi,\psi)\in X(0,\infty)$ satisfying
\begin{equation}\label{priori}
\begin{split}
&\|\phi\|_{1,w_1}^2+\|\psi\|_{1,w_2}^2+\|\phi_{xx}\|^2+\|\psi_{xx}\|^2
\\&\ \ +\int_0^t(\|\phi_x(\tau)\|_{1,w_1}^2+\|\psi_x(\tau)\|_{1,w_2}^2+\|\phi_{xxx}(\tau)\|^2+\|\psi_{xxx}(\tau)\|^2)d\tau\leq
CN^2(0)
\end{split}
\end{equation}
for any $t\in [0,\infty)$.
\end{proposition}

The local existence of solutions to system \eqref{newfai}-\eqref{newbound2} is standard (e.g., see \cite{Nishida78}). To prove Proposition
\ref{global-existence}, we
only need to derive the following {\it a priori} estimates.

\begin{proposition}\label{mingti-1}
Assume that the conditions of Proposition \ref{global-existence}
hold, and that $(\phi,\psi)\in X(0,T)$ is a solution of system
\eqref{newfai}-\eqref{newbound2}  for some
constant $T>0$. Then there is a positive constant $\varepsilon_1>0$,
independent of $T$, such that if $N(t)\leq \varepsilon_1$ for any
$0\leq t\leq T$, then $(\phi,\psi)$ satisfies $\eqref{priori}$ for
any $0\leq t\leq T$.
\end{proposition}
We first establish the basic $L^2$ estimate.

\begin{lemma}\label{lem-large1}
If $N(t)\ll 1$, then there exists a constant $C>0$ such that
\begin{equation}\label{3.15}
\begin{split}\int_0^\infty \left(\frac{\phi^{2}}{U}+ W^{1-m}\psi^{2}\right)
+\int_0^t \int_0^\infty \left(\frac{\phi_{x}^{2}}{U}+
W^{1-m}\psi_{x}^{2}+U\psi^2\right)\leq
C(\|\phi_{0}\|^2_{w_1}+\|\psi_{0}\|^2_{w_2}).
\end{split}
\end{equation}

\end{lemma}

\begin{proof}
Multiplying the first equation of \eqref{newfai} by $\frac{\phi}{U}$
and  the second one  by $\chi  W^{1-m}\psi$, integrating the
resulting equations in $x$, and using the Taylor expansion to get
\begin{equation*}
1-e^{-(m-1)\psi}=(m-1)\psi-\sum_{n=2}^{\infty}\frac{(1-m)^{n}\psi^{n}}{n!},
\end{equation*}
we have
\begin{equation}\label{3.16}
\begin{split}&\frac{1}{2}\frac{d}{dt}\int_0^\infty \left(\frac{\phi^{2}}{U}+\chi W^{1-m}\psi^{2}\right)+\int_0^\infty \frac{\phi_{x}^{2}}{U}+\chi\varepsilon\int_0^\infty  W^{1-m}\psi_{x}^{2}\\&\quad-\int_0^\infty \frac{\phi^{2}}{2}\left[\left(\frac{1}{U}\right)_{xx}-\left(\frac{\chi V}{U}\right)_{x}\right]-\chi\int_0^\infty \psi^{2}\left[\frac{\varepsilon}{2}( W^{1-m})_{xx}+\varepsilon(V W^{1-m})_{x}+(1-m)U\right]\\&=\chi\int_0^\infty\frac{\phi\phi_{x}\psi_{x}}{U}-\chi\varepsilon\int_0^\infty W^{1-m}\psi\psi_{x}^{2}+\chi(1-m)\int_0^\infty \phi_{x}\psi^{2}\\&\quad+\chi\int_0^\infty(U\psi+\phi_{x}\psi)\sum_{n=2}^{\infty}\frac{(1-m)^{n}\psi^{n}}{n!}.
\end{split}\end{equation}
A direct calculation by \eqref{ssuv-2} and \eqref{diff-equ} yields
\begin{equation}\label{3.17}
\left(\frac{1}{U}\right)_{xx}-\left(\frac{\chi V}{U}\right)_{x}=0,
\end{equation}
and
\begin{equation}\label{caclu}
\begin{split}
&\frac{\varepsilon}{2}( W^{1-m})_{xx}+\varepsilon(V W^{1-m})_{x}+(1-m)U\\
&=-\frac{\varepsilon}{2}(1-m)m W^{-m-1}
W_{x}^{2}+\frac{\varepsilon}{2}(1-m) W^{-m} W_{xx}\\&\quad
+\varepsilon V_x W^{1-m}+\varepsilon(1-m)V W^{-m} W_{x}+(1-m)U
\\&
=-\frac{\varepsilon}{2}(1-m)m W^{-m-1} W_{x}^{2}
+\frac{1-3m}{2}U+\varepsilon V^{2} W^{1-m}+\varepsilon(1-m)V W^{-m}
W_{x}.
\end{split}\end{equation}
To estimate \eqref{caclu}, for convenience, we set
\begin{equation}\label{notion}
\theta:=\lambda(\chi+1-m)>0,\ \beta:=\varepsilon(\chi+m+1)b^{1-m}>0,\
r:=\frac{\chi+m-1}{2}>0.
\end{equation}
Then by \eqref{U} and \eqref{C}, we have
\begin{equation}\label{3.19}
\begin{split}
&U(x)=\frac{\theta^{2}}{2\beta}\left(1+\frac{\theta r}{\beta}x\right)^{-\frac{\chi}{r}},\  W(x)=b\left(1+\frac{\theta r}{\beta}x\right)^{-\frac{1}{r}},\\
&W_{x}(x)=-\frac{b\theta}{\beta}\left(1+\frac{\theta r}{\beta}x\right)^{-\frac{1}{r}-1},\
V(x)=-\frac{W_x}{W}=\frac{\theta}{\beta}\left(1+\frac{\theta
r}{\beta}x\right)^{-1},\
\end{split}
\end{equation}
Substituting \eqref{3.19} into \eqref{caclu} gives
\begin{equation}\begin{split}\label{s00}
\text{RHS of } \eqref{caclu}&=\frac{\theta^{2}}{\beta}\left(\frac{m^{2}+m}{2(\chi+m+1)}+\frac{1-3m}{4}\right)\left(1+\frac{\theta r}{\beta}\cdot x\right)^{-\frac{\chi}{r}}\\
&=\frac{1-m^2+\chi(1-3m)}{2(\chi+m+1)}\cdot U\\
&\leq-\frac{\chi}{\chi+m+1}\cdot U
\end{split}\end{equation}
where we have used $m\geq1$.
Next we estimate the terms on the RHS of \eqref{3.16}. With the fact $\frac{|U_{x}|}{U}\leq \frac{\a\theta}{\b}$, we derive that
\begin{equation*}
\frac{\phi^{2}}{U}=\int_0^x \left(\frac{\phi^{2}}{U}\right)_{x}
=\int_0^x
\left(\frac{2\phi\phi_{x}}{U}-\frac{\phi^{2}U_{x}}{U^2}\right) \leq
C\bigg(\int_0^\infty \frac{\phi^{2}}{U}+\int_0^\infty
\frac{\phi_{x}^{2}}{U}\bigg) \leq CN^{2}(t),
\end{equation*}
and hence have
\begin{equation*}
\frac{|\phi|}{\sqrt{U}}\leq CN(t).
\end{equation*}
Notice that $ W^{1-m}(x)> b^{1-m}$ over $(0,\infty)$ when $m\geq1$. Then it follows that
\begin{equation*}\begin{split}
\chi\int_0^\infty \frac{|\phi\phi_{x}\psi_{x}|}{U}
&=\chi\int_0^\infty
\frac{|\phi|}{\sqrt{U}}\cdot\frac{|\phi_{x}|}{\sqrt{U}}\cdot|\psi_{x}|\\
&\leq \frac{CN(t)}{\varepsilon b^{1-m}}\int_0^\infty
\frac{\phi_{x}^{2}}{U}+\chi\varepsilon
N(t)b^{1-m}\int_0^\infty \psi_{x}^{2}\\&\leq CN(t)\int_0^\infty
\frac{\phi_{x}^{2}}{U}+\chi \varepsilon N(t)\int_0^\infty
W^{1-m}\psi_{x}^{2}.\end{split}
\end{equation*}
By Cauchy-Schwarz inequality and the fact $\|\psi(\cdot,t)\|_{L^\infty}\leq N(t)$, one has
\begin{eqnarray*}
\begin{aligned}
&\chi(1-m)\int_0^\infty \phi_{x}\psi^{2}\leq N(t)\int_0^\infty
\frac{\phi_{x}^{2}}{U}+CN(t)\int_0^\infty U\psi^{2},\\
&-\chi\varepsilon\int_0^\infty W^{1-m}\psi\psi_{x}^{2}\leq\chi \varepsilon N(t)\int_0^\infty W^{1-m}\psi_x^{2}.
\end{aligned}
\end{eqnarray*}
Furthermore, noting $e^{m-1}=\sum\limits_{n=0}^{\infty}\frac{(1-m)^{n}}{n!}$, if
$N(t)<1$ and hence $\|\psi(\cdot,t)\|_{L^\infty}\leq1$, we have
\begin{equation*}
\left|\sum_{n=2}^{\infty}\frac{(1-m)^{n}\psi^{n}}{n!}\right|
\leq(1-m)^{2}\psi^{2}\sum_{n=2}^{\infty}\frac{(1-m)^{n-2}}{n!}\leq(1-m)^{2}\psi^{2}e^{m-1}.
\end{equation*}
Hence,
\begin{equation}\begin{split}\label{s33}
\left|\int_0^\infty
(U+\phi_{x})\psi\sum_{n=2}^{\infty}\frac{(1-m)^{n}\psi^{2}}{n!}\right|
&\leq C(m-1)^{2}\int_0^\infty (U+|\phi_{x}|)|\psi|^{3}\\& \leq
CN(t)\int_0^\infty U\psi^{2}+N(t)\int_0^\infty
\frac{\phi_{x}^{2}}{U}.\end{split}
\end{equation}
Now substituting \eqref{s00}-\eqref{s33} into \eqref{3.16}, we have
\begin{equation*}
\begin{split}&\frac{1}{2}\int_0^\infty \left(\frac{\phi^{2}}{U}+\chi  W^{1-m}\psi^{2}\right)
+(1-CN(t))\int_0^t\int_0^\infty \frac{\phi_{x}^{2}}{U}\\
&+\chi\varepsilon(1-2N(t))\int_0^t \int_0^\infty  W^{1-m}\psi_{x}^{2}
+\left(\frac{\chi^2}{\chi+m+1}-CN(t)\right)\int_0^t\int_0^\infty U\psi^{2}\\
&\leq \frac{1}{2}\int_0^\infty \left(\frac{\phi_{0}^{2}}{U}+\chi
W^{1-m}\psi_{0}^{2}\right).
\end{split}
\end{equation*}
Therefore, \eqref{3.15} holds provided that $N(t)\ll1$.
\end{proof}

We next establish the $H^1$ estimate.
\begin{lemma}\label{H1}
If $N(t)\ll1$, then the solution of \eqref{newfai}-\eqref{newbound2} satisfies
\begin{equation}\label{3.20}
\begin{split}\int_0^\infty \left(\frac{\phi_x^{2}}{U}+ W^{1-m}\psi_x^{2}\right)
+\int_0^t \int_0^\infty \left(\frac{\phi_{xx}^{2}}{U}+
W^{1-m}\psi_{xx}^{2}\right) \leq
C(\|\phi_{0}\|^2_{1,w_1}+\|\psi_{0}\|^2_{1,w_2})
\end{split}
\end{equation}
where $C>0$ is a constant independent of $t$.
\end{lemma}
\begin{proof}
Multiplying the first equation of \eqref{newfai} by
$\frac{\phi_{xx}}{U}$, integrating the resultant equation in $x$,
and noting
\begin{equation*}
\begin{split}\frac{\phi_{t}\phi_{xx}}{U}
=\left(\frac{\phi_{t}\phi_{x}}{U}\right)_{x}-\left(\frac{\phi_{x}^{2}}{2U}\right)_{t}+\phi_t\phi_{x}\cdot\frac{U_{x}}{U^{2}},
\end{split}
\end{equation*}
we get
\begin{equation}\label{nn}
\begin{split}
\frac{d}{dt}\int_0^\infty \frac{\phi_x^{2}}{2U} +\int_0^\infty
\frac{\phi_{xx}^{2}}{U}&=-\chi\int_0^\infty\left(\frac{V\phi_{x}}{U}+\psi_{x}\right)\phi_{xx}
-\chi\ii\frac{\phi_{x}\psi_{x}\phi_{xx}}{U}+\ii\phi_t\phi_{x}\cdot\frac{U_{x}}{U^{2}}.
\end{split}\end{equation}
By Young's inequality, the following inequalities hold
$$\chi\left|\left(\frac{V\phi_{x}}{U}+\psi_{x}\right)\phi_{xx}\right|\leq \frac{\phi_{xx}^{2}}{2U}+\frac{\chi^{2}V^{2}\phi_{x}^{2}}{U}+\chi^{2}U\psi_{x}^{2},$$
$$\chi \left|\frac{\phi_{x}\psi_{x}\phi_{xx}}{U}\right|\leq\frac{N(t)\phi_{xx}^{2}}{U}+\frac{N(t)\chi^{2}\phi_{x}^{2}}{4U},$$
where we have used the fact that $\|\psi_x(\cdot,t)\|_{L^\infty}\leq
N(t)$. Similarly, noting $\frac{|U_{x}|}{U^{2}}\leq
\frac{\a\theta}{\b U}$, we have
\[\begin{split}\phi_t\phi_{x}\cdot\frac{U_{x}}{U^{2}}&=(\phi_{xx}+\chi V\phi_{x}+\chi U\psi_{x}+\chi\phi_{x}\psi_{x})\phi_{x}\cdot\frac{U_{x}}{U^{2}}\\
&\leq\frac{\phi_{xx}^{2}}{4U}+\frac{C(1+N(t))\phi_{x}^{2}}{U}+C\phi_{x}^{2}+C\psi_{x}^{2}.\end{split}\]
Thus, it follows from \eqref{nn} that
\begin{equation}\label{newe}
\frac{d}{dt}\int_0^\infty
\frac{\phi_{x}^{2}}{2U}+\left(\frac{1}{4}-N(t)\right)\int_0^\infty
\frac{\phi_{xx}^{2}}{U}\leq C\left(\int_0^\infty
\frac{\phi_{x}^{2}}{U}+\int_0^\infty \psi_{x}^{2}\right),
\end{equation}
which, along with \eqref{3.15} and the fact $W^{1-m}> b^{1-m}$ over $(0,\infty)$ for $m>1$,  leads to
\begin{equation}\label{3.23}
\begin{split}
\int_0^\infty \frac{\phi_x^{2}}{U} +\int_0^t\int_0^\infty
\frac{\phi_{xx}^{2}}{U}\leq C\int_0^\infty
\left(\frac{\phi_{0x}^{2}}{U}+\frac{\phi_{0}^{2}}{U}+
W^{1-m}\psi_{0}^{2}\right).
\end{split}\end{equation}
Multiplying the second equation of \eqref{newfai} by $
W^{1-m}\psi_{xx}$, and using the following inequality
\begin{equation*}
\begin{split}\psi_{t} W^{1-m}\psi_{xx}=( W^{1-m}\psi_{t}\psi_{x})_{x}-(1-m) W^{-m} W_{x}\psi_{t}\psi_{x}-\left( W^{1-m}\frac{\psi_{x}^{2}}{2}\right)_{t},
\end{split}
\end{equation*}
we get
\begin{equation}\label{3.24}
\begin{split}&\frac{1}{2}\frac{d}{dt}\int_0^\infty  W^{1-m}\psi_x^{2}+\varepsilon\int_0^\infty  W^{1-m}\psi_{xx}^{2}\\&=
\ii(2\varepsilon V W^{1-m}\psi_x-\phi_x)\psi_{xx}+\ii(1-e^{-(m-1)\psi})(U+\phi_x)\psi_{xx}\\
&\quad+\varepsilon\ii W^{1-m}\psi_x^{2}\psi_{xx}-(1-m)\int_0^\infty
W^{-m} W_{x}\psi_{t}\psi_{x}.
\end{split}\end{equation}
Furthermore Young's inequality gives rise to the following estimate:
\begin{equation*}
|(2\varepsilon V W^{1-m}\psi_x-\phi_x)\psi_{xx}|\leq
\frac{\varepsilon  W^{1-m}\psi_{xx}^{2}}{4}+8\varepsilon V^2
W^{1-m}\psi_{x}^{2}+\frac{2 W^{m-1}\phi_{x}^{2}}{\varepsilon}.
\end{equation*}
Since $|1-e^{-(m-1)\psi}|\leq C(m-1)|\psi|$ if $N(t)<1$ by Taylor's theorem, we have
\begin{equation*}
\begin{split}&|(1-e^{-(m-1)\psi})(U+\phi_x)\psi_{xx}|\\&\leq C(m-1)(U+|\phi_x|)|\psi\psi_{xx}|\\
&\leq\frac{(\varepsilon+N(t))}{4} W^{1-m}\psi_{xx}^{2}+C(m-1)^2 W^{m-1}U^2\psi^{2}+C(m-1)^2N(t) W^{m-1}\phi_x^2\\
&\leq\frac{(\varepsilon+N(t))}{4}
W^{1-m}\psi_{xx}^{2}+CU\psi^{2}+CN(t)\frac{\phi_x^2}{U}
\end{split}
\end{equation*}
where in view of \eqref{3.19} we have used the fact $ W^{m-1}\leq b^{m-1}$ and
\begin{equation}\label{3.25}
 W^{m-1}U=\frac{b^{m-1}\theta^{2}}{2\beta}\left(1+\frac{\theta r}{\beta}x\right)^{-2}.
\end{equation}
Similarly, since $\|\psi_x(\cdot,t)\|_{L^\infty}\leq N(t)$, we get
\begin{equation*}
-\varepsilon W^{1-m}\psi_{x}^{2}\psi_{xx}\leq \frac{N(t)
W^{1-m}}{2}\psi_{xx}^{2}+\varepsilon^2 N(t)  W^{1-m}\psi_{x}^{2},
\end{equation*}
and
\[\begin{split}-(1-m) W^{-m} W_{x}\psi_{t}\psi_{x}&\leq\frac{ W^{1-m}}{4\varepsilon}\psi_{t}^{2}+\varepsilon(1-m)^2 W^{-1-m} W_{x}^2\psi_{x}^2\\&
\leq\frac{\varepsilon W^{1-m}}{4}\psi_{xx}^{2}+C(W^{1-m}\psi_{x}^2+U\psi^2+\phi_x^2),\end{split}\] where we have used
the second equation of \eqref{newfai} and $ W^{-2}
W_x^2=\frac{\theta^2}{\beta^2}\left(1+\frac{\theta
r}{\beta}x\right)^{-2}\leq\frac{\theta^2}{\beta^2}$. Now integrating
\eqref{3.24} in $t$, we arrive at
\begin{equation*}
\begin{split}&\frac{1}{2}\int_0^\infty  W^{1-m}\psi_x^{2}+\frac{1}{4}\left(\varepsilon-N(t)\right)\int_0^t\int_0^\infty W^{1-m}\psi_{xx}^{2}\\&\leq\frac{1}{2}\int_0^\infty  W^{1-m}\psi_{0x}^{2}+C\int_0^t\int_0^\infty ( W^{1-m}\psi_{x}^{2}+\phi_x^2+U\psi^2)
\end{split}\end{equation*}
which by \eqref{3.15} further gives
\begin{equation}\label{3.26}
\int_0^\infty  W^{1-m}\psi_{x}^{2}+\varepsilon\int_0^t \int_0^\infty
W^{1-m}\psi_{xx}^{2}\leq C\int_0^\infty \left(
W^{1-m}\psi_{0x}^{2}+\frac{\phi_{0}^{2}}{U}+
W^{1-m}\psi_{0}^{2}\right),
\end{equation}
if $N(t)\ll1$. The desired \eqref{3.20} follows from \eqref{3.23}
and \eqref{3.26}.
\end{proof}

The $H^2$ estimate is as follows.
\begin{lemma}\label{H2}
If $N(t)\ll1$, then it follows that
\begin{equation}\label{3.27}
\begin{split}&\int_0^\infty \left(\phi_{xx}^{2}+\psi_{xx}^{2}\right)
+\int_0^t \int_0^\infty
\left(\phi_{xxx}^{2}+\psi_{xxx}^{2}\right)\\&\leq
C(\|\phi_{0xx}\|^2+\|\psi_{0xx}\|^2+\|\phi_{0}\|^2_{1,w_1}+\|\psi_{0}\|^2_{1,w_2})
\end{split}
\end{equation}
where $C>0$ is a constant independent of $t$.
\end{lemma}

\begin{proof} By \eqref{newfai}, \eqref{3.15} and \eqref{3.20}, it is easy to see that
\begin{equation}\label{3.28}
\begin{split}\int_0^t \int_0^\infty \phi_t^2&\leq C\int_0^t \int_0^\infty (\phi_{xx}^{2}+V^2\phi_{x}^{2}+U^2\psi_{x}^{2}+\phi_{x}^{2}\psi_{x}^{2})\\
&\leq C(\|\phi_{0}\|^2_{1,w_1}+\|\psi_{0}\|^2_{1,w_2}),\end{split}
\end{equation}and
\begin{equation}\label{3.29}
\begin{split}
\int_0^t \int_0^\infty \psi_t^2&\leq C\int_0^t \int_0^\infty (\psi_{xx}^{2}+V^2\psi_{x}^{2}+U^2 W^{2(m-1)}\psi^2+ W^{2(m-1)}\phi_{x}^{2}+\psi_{x}^{2}+ W^{2(m-1)}\phi_{x}^{2}\psi^{2})\\
&\leq C(\|\phi_{0}\|^2_{1,w_1}+\|\psi_{0}\|^2_{1,w_2}).\end{split}
\end{equation}
Differentiating \eqref{newfai} with respect to $t$ leads to
\begin{eqnarray}\label{t-equ}
\left\{
\begin{array}{lll}
\phi_{tt}=&\phi_{txx}+\chi V\phi_{tx}+\chi U\psi_{tx}+\chi\phi_{tx}\psi_{x}+\chi\phi_{x}\psi_{tx},\\[2mm]
\psi_{tt}=&\varepsilon\psi_{txx}-2\varepsilon V\psi_{tx}-(m-1)U
W^{m-1}e^{-(m-1)\psi}\psi_t-2\varepsilon\psi_{x}\psi_{tx}
\\[2mm]
&-(m-1) W^{m-1}e^{-(m-1)\psi}\psi_t\phi_{x}+ W^{m-1}e^{-(m-1)\psi}\phi_{tx}.
\end{array}
\right.
\end{eqnarray}
Multiplying the first equation of \eqref{t-equ} by $\phi_t$ and
integrating it in $x$, we get
\begin{equation}\label{3.30}
\begin{split}
&\frac{1}{2}\frac{d}{dt}\int_0^\infty \phi_t^2+\ii\phi_{tx}^2\\&=\chi\ii(V\phi_{tx}+U\psi_{tx}+\phi_{tx}\psi_x+\phi_x\psi_{tx})\phi_t\\
&\leq\left(\frac{1}{4}+N(t)\right)\ii(\phi_{tx}^2+\varepsilon\psi_{tx}^2)+C\ii(V^2+U^2+N(t))\phi_t^2,
\end{split}
\end{equation}
where we have used $\|\psi_x(\cdot,t)\|_{L^\infty}\leq N(t)$ and $\|\phi_x(\cdot,t)\|_{L^\infty}\leq N(t)$ in the above inequality.
Similarly, multiplying the second equation of \eqref{t-equ} by
$\psi_t$ and integrating it in $x$,
\begin{equation}\label{3.31}
\begin{split}
&\frac{1}{2}\frac{d}{dt}\int_0^\infty \psi_t^2+\varepsilon\ii\psi_{tx}^2-\varepsilon\ii V_x\psi_t^2+(m-1)\ii U W^{m-1}e^{-(m-1)\psi}\psi_t^2\\&=\ii( W^{m-1}e^{-(m-1)\psi}\phi_{tx}-2\varepsilon\psi_{x}\psi_{tx}-(m-1) W^{m-1}e^{-(m-1)\psi}\psi_t\phi_{x})\psi_t\\
&\leq \ii\left(\frac{1}{4}\phi_{tx}^2+\varepsilon
N(t)\psi_{tx}^2\right)+C\ii\psi_t^2.
\end{split}
\end{equation}
Thus, combining \eqref{3.30} with \eqref{3.31}, and noticing $V_x<0$,
$m\geq1$ and $N(t)\ll1$, we have
\begin{equation}\label{3.32}
\begin{split}
&\int_0^\infty (\phi_t^2+\psi_t^2)+\int_0^t\ii(\phi_{tx}^2+\psi_{tx}^2)\\
&\leq C\ii(\phi_{0xx}^2+\phi_{0x}^2+\psi_{0x}^2+\psi_{0xx}^2)+ C\int_0^t\ii(\phi_t^2+\psi_t^2)\\
&\leq
C(\|\phi_{0xx}\|^2+\|\psi_{0xx}\|^2+\|\phi_{0}\|^2_{1,w_1}+\|\psi_{0}\|^2_{1,w_2})
\end{split}
\end{equation}
where we have used \eqref{3.28}, \eqref{3.29} and the compatible
condition of the initial data. Using   \eqref{newfai} again, we also
get
\begin{equation}\label{3.33}
\int_0^\infty \phi_{xx}^2\leq
C\int_0^\infty(\phi_t^2+\phi_x^2+\psi_x^2)\leq
C(\|\phi_{0xx}\|^2+\|\psi_{0xx}\|^2+\|\phi_{0}\|^2_{1,w_1}+\|\psi_{0}\|^2_{1,w_2})
\end{equation}
and
\begin{equation*}
\int_0^\infty \psi_{xx}^2\leq
C\int_0^\infty(\psi_t^2+\psi_x^2+\phi_x^2)\leq
C(\|\phi_{0xx}\|^2+\|\psi_{0xx}\|^2+\|\phi_{0}\|^2_{1,w_1}+\|\psi_{0}\|^2_{1,w_2}).
\end{equation*}
Differentiating the first equation of \eqref{newfai} in $x$ yields
\begin{equation*}
\phi_{xxx}=\phi_{tx}-\chi V\phi_{xx}-\chi V_x\phi_{x}-\chi
U\psi_{xx}-\chi
U_x\psi_{x}-\chi\phi_{xx}\psi_{x}-\chi\phi_{x}\psi_{xx},
\end{equation*}
which in combination with \eqref{3.28}, \eqref{3.29} and
\eqref{3.32} leads to
\begin{equation*}
\int_0^t\ii\phi_{xxx}^2\leq
C(\|\phi_{0xx}\|^2+\|\psi_{0xx}\|^2+\|\phi_{0}\|^2_{1,w_1}+\|\psi_{0}\|^2_{1,w_2}).
\end{equation*}
Similarly, differentiating the second equation of \eqref{newfai} in
$x$, and using  \eqref{3.28}, \eqref{3.29} and \eqref{3.32}, we have
\begin{equation}\label{3.36}
\int_0^t\ii\psi_{xxx}^2\leq
C(\|\phi_{0xx}\|^2+\|\psi_{0xx}\|^2+\|\phi_{0}\|^2_{1,w_1}+\|\psi_{0}\|^2_{1,w_2}).\end{equation}
The desired estimate \eqref{3.27} follows from
\eqref{3.33}-\eqref{3.36}.
\end{proof}

\begin{remark}
Proposition \ref{mingti-1} is a consequence of Lemmas \ref{lem-large1}, \ref{H1} and \ref{H2}.
\end{remark}

\subsection{Case $0\leq m<1$}\label{sec.3.2}
As in the case $m\geq1$, we look for solutions of system
\eqref{newfai} with \eqref{newbound1} and \eqref{newbound2} in the
space
\begin{equation*}
\begin{split}
Y(0,T):=\{&(\phi,\psi)(x,t)\big|\phi\in C([0,T]; H^2\cap
H^1_{w_3}),\phi_x\in L^2((0,T);H^2\cap H^1_{w_3}),\\&\psi\in C([0,T];
H^2),\psi_x\in L^2((0,T);H^2)\},
\end{split}\end{equation*}
for $T\in(0,+\infty]$, where $w_3=W^{m-1}/U$. Set
\begin{equation*}
N(t):=\sup_{\tau\in[0,t]}(\|\phi(\cdot,\tau)\|_{1,w_3}+\|\phi_{xx}(\cdot,\tau)\|+\|\psi(\cdot,\tau)\|_{2}).
\end{equation*}

\begin{proposition}\label{global existence-2} Assume that $0\leq
m<1$ and that $\chi$ is large enough. There exists a
constant $\delta_2$, such that if $N(0)\leq\delta_2$, then system
\eqref{newfai}-\eqref{newbound2} has a
unique global solution $(\phi,\psi)\in Y(0,\infty)$ satisfying
\begin{equation}\label{priori-2}
\begin{split}
\|\phi\|_{1,w_3}^2+\|\phi_{xx}\|^2+\|\psi\|_{2}^2
+\int_0^t(\|\phi_x(\tau)\|_{1,w_3}^2+\|\phi_{xx}(\tau)\|^2+\|\psi_{x}(\tau)\|_2^2)d\tau\leq
CN^2(0)
\end{split}
\end{equation}
for any $t\in [0,\infty)$.
\end{proposition}

To prove Proposition \ref{global existence-2}, it suffices to derive the
following a priori estimates.

\begin{proposition}\label{mingti-2}
Under the same assumptions of Proposition \ref{global existence-2},
if $(\phi,\psi)\in Y(0,T)$ is a solution of system \eqref{newfai}-\eqref{newbound2} for a constant $T>0$,
then there is a positive constant $\varepsilon_2>0$, independent of
$T$, such that if $N(t)\leq \varepsilon_2$ for any $0\leq t\leq T$,
then $(\phi,\psi)$ satisfies $\eqref{priori-2}$ for any $0\leq t\leq
T$.
\end{proposition}

The following Hardy inequality plays an important role in
establishing the a priori estimates.
\begin{lemma}[Hardy inequality]\label{hardy}
If $f\in H_0^1(0,\infty)$, then for $j\neq-1$, it holds that
\begin{equation*}
\ii(1+kx)^jf^2(x)dx\leq\frac{4}{(j+1)^2k^2}\ii(1+kx)^{j+2}f_x^2(x)dx,
\end{equation*}
where $k>0$ is a constant.
\end{lemma}
\begin{proof}
Since $C_0^\infty(0,\infty)$ is dense in $H_0^1(0,\infty)$, by density argument (cf. \cite[Section 50.3]{QS}), we only consider $f\in
C_0^\infty(0,\infty)$. Then by Cauchy-Schwarz inequality, for
$j\neq-1$, we have
\[\begin{split}
\ii(1+kx)^jf^2(x)dx&=\frac{1}{(j+1)k}\ii f^2(x)d((1+kx)^{j+1})\\
&=\frac{2}{(j+1)k}\ii
(1+kx)^{j+1}f(x)f_x(x)dx\\&\leq\frac{2}{|(j+1)k|}\left(\ii
(1+kx)^jf^2(x)dx\right)^{\frac{1}{2}}\left(\ii
(1+kx)^{j+2}f_x^2(x)dx\right)^{\frac{1}{2}}\end{split}\]
which complete the proof.
\end{proof}

We now derive the $L^2$ estimate.
\begin{lemma}\label{lem-small1}
If $N(t)\ll1$, then there is a constant $C>0$ independent of $t$ such that the solution of system \eqref{newfai}-\eqref{newbound2} satisfies
\begin{equation}\label{3.38}
\begin{split}&\int_0^\infty \left(\frac{ W^{m-1}\phi^{2}}{U}+\psi^{2}\right)
+\int_0^t \int_0^\infty \left(\frac{
W^{m-1}\phi_{x}^{2}}{U}+\psi_{x}^{2}+U W^{m-1}\psi^2\right)
\\&\leq C\int_0^\infty \left(\frac{ W^{m-1}\phi_{0}^{2}}{U}+ \psi_{0}^{2}\right).
\end{split}
\end{equation}
\end{lemma}

\begin{proof}
Multiplying the first equation of \eqref{newfai} by $\frac{
W^{m-1}\phi}{U}$ and  the second one  by $\chi \psi$, integrating
the resultant equations in $x$, we have
\begin{equation}\label{3.39}
\begin{split}&\frac{1}{2}\frac{d}{dt}\int_0^\infty \left(\frac{ W^{m-1}\phi^{2}}{U}+\chi\psi^{2}\right)+\int_0^\infty \frac{ W^{m-1}\phi_{x}^{2}}{U}+\chi\varepsilon\int_0^\infty \psi_{x}^{2}+\chi\ii( W^{m-1})_x\phi\psi\\&\quad-\int_0^\infty \frac{\phi^{2}}{2}\left[\left(\frac{ W^{m-1}}{U}\right)_{xx}-\left(\frac{\chi  W^{m-1}V}{U}\right)_{x}\right]-\chi\int_0^\infty \psi^{2}\left[\varepsilon V_{x}+(1-m)U W^{m-1}\right]\\&=\chi\int_0^\infty\frac{ W^{m-1}\phi\phi_{x}\psi_{x}}{U}-\chi\varepsilon\int_0^\infty \psi\psi_{x}^{2}+\chi(1-m)\int_0^\infty W^{m-1}\phi_{x}\psi^{2}\\&\quad+\chi\int_0^\infty W^{m-1}(U\psi+\phi_{x}\psi)\sum_{n=2}^{\infty}\frac{(1-m)^{n}\psi^{n}}{n!}.
\end{split}\end{equation}
A direct calculation by \eqref{3.17} and \eqref{3.19} yields
\begin{equation}\label{3.40}
\chi( W^{m-1})_x=\frac{(1-m)\chi\theta b
^{m-1}}{\b}\left(1+\frac{\theta r}{\beta}x\right)^{\frac{1-m}{r}-1},
\end{equation}
\begin{equation*}
\begin{split}
&-\frac{1}{2}\left[\left(\frac{W^{m-1}}{U}\right)_{xx}-\left(\frac{\chi W^{m-1}V}{U}\right)_{x}\right]\\
&=-\frac{1}{2U}\left[(W^{m-1})_{xx}-\chi( W^{m-1})_xV-2(W^{m-1})_x\frac{U_x}{U}\right]
-\frac{ W^{m-1}}{2}\left[\left(\frac{1}{U}\right)_{xx}-\left(\frac{\chi V}{U}\right)_{x}\right]\\
&=-\frac{1}{2U}\left[( W^{m-1})_{xx}+\chi( W^{m-1})_xV\right]\\
&=\frac{(1-m)b^{m-1}}{\b}((m-1+r)-\k)\left(1+\frac{\theta r}{\beta}x\right)^{\frac{2-2m}{r}},
\end{split}
\end{equation*}
and
\begin{equation*}
\begin{split}
&-\chi\left[\varepsilon V_{x}+(1-m)U
W^{m-1}\right]=\frac{\chi\theta^2}{\b}\left(\frac{\varepsilon
r}{\b}-\frac{(1-m)b^{m-1}}{2}\right)\left(1+\frac{\theta
r}{\beta}x\right)^{-2}.
\end{split}
\end{equation*}
By Lemma \ref{hardy} with \eqref{notion}, we have
\begin{equation}\label{3.43}
\begin{split}
\frac{1}{2}\int_0^\infty \frac{ W^{m-1}\phi_{x}^{2}}{U}&=\frac{\beta b^{m-1}}{\theta^2}\ii\left(1+\frac{\theta r}{\beta}x\right)^{\frac{\chi+1-m}{r}}\phi_{x}^{2}\\&\geq\frac{b^{m-1}}{4\beta}(\chi+1-m-r)^2\ii\left(1+\frac{\theta r}{\beta}x\right)^{\frac{2-2m}{r}}\phi^2,\\
\frac{\chi\varepsilon}{2}\int_0^\infty
\psi_{x}^{2}&\geq\frac{\chi\varepsilon\theta^2r^2}{8\beta^2}\ii\left(1+\frac{\theta
r}{\beta}x\right)^{-2}\psi^2.
\end{split}\end{equation}
Moreover, a direct calculation in view of \eqref{notion} gives
\[\begin{split}
&\frac{(1-m)b^{m-1}}{\b}((m-1+r)-\k)+\frac{b^{m-1}}{4\beta}(\chi+1-m-r)^2\\
&=\frac{b^{m-1}}{16\b}\left[8(1-m)(\a+3m-3)-16\a(1-m)+(\a-3m+3)^2\right]\\
&=\frac{b^{m-1}}{16\b}[\a+3(1-m)][\a-5(1-m)]\\&\triangleq B_1,
\end{split}\]
and
\[\begin{split}
&\frac{\chi\theta^2}{\b}\left(\frac{\varepsilon r}{\b}-\frac{(1-m)b^{m-1}}{2}\right)+\frac{\chi\varepsilon\theta^2r^2}{8\beta^2}\\
&=\frac{\k\theta^2 b^{m-1}}{2\b}\left[\frac{(\a+m-1)^2}{16(\a+m+1)}+m-\frac{2}{\a+m+1}\right]\\
&\triangleq B_2.
\end{split}\]
Now substituting \eqref{3.40}-\eqref{3.43} into \eqref{3.39}, and
noting when $\chi\gg1$, there exists a constant $C_1>0$ such that
\[\begin{split}
&B_1\left(1+\frac{\theta r}{\beta}x\right)^{\frac{2-2m}{r}}\phi^2+B_2\left(1+\frac{\theta r}{\beta}x\right)^{-2}\psi^2+\frac{(1-m)\chi\theta b ^{m-1}}{\b}\left(1+\frac{\theta r}{\beta}x\right)^{\frac{1-m}{r}-1}\phi\psi\\
&\geq C_1\left(\left(1+\frac{\theta
r}{\beta}x\right)^{\frac{2-2m}{r}}\phi^2+\left(1+\frac{\theta
r}{\beta}x\right)^{-2}\psi^2\right),
\end{split}\]
one can see that
\begin{equation}\begin{split}\label{3.44}
\text{LHS of }
\eqref{3.39}\geq&\frac{1}{2}\frac{d}{dt}\int_0^\infty \left(\frac{
W^{m-1}\phi^{2}}{U}+\chi\psi^{2}\right)+\frac{1}{2}\int_0^\infty
\left(\frac{W^{m-1}\phi_{x}^{2}}{U}+
\chi\varepsilon\psi_{x}^{2}\right)\\&+C_1\ii\left(\left(1+\frac{\theta
r}{\beta}x\right)^{\frac{2-2m}{r}}\phi^2+\left(1+\frac{\theta
r}{\beta}x\right)^{-2}\psi^2\right).\end{split}
\end{equation}

The RHS of \eqref{3.39} can be estimated as follows. It is easy to
see that
\begin{equation*}\begin{split}
\frac{ W^{m-1}\phi^{2}}{U}=\int_0^x \left(\frac{
W^{m-1}\phi^{2}}{U}\right)_{x}
&=\int_0^x \left(\frac{2 W^{m-1}\phi\phi_{x}}{U}+\frac{( W^{m-1})_x\phi^2}{U}-\frac{\phi^{2} W^{m-1}U_{x}}{U^2}\right)\\
&\leq C\int_0^\infty \frac{ W^{m-1}\phi^{2}}{U}+\int_0^\infty \frac{
W^{m-1}\phi_{x}^{2}}{U}\\& \leq CN^{2}(t),\end{split}
\end{equation*}
which implies $\frac{\sqrt{ W^{m-1}}}{\sqrt{U}}|\phi|\leq CN(t).$
Hence
\begin{equation}\begin{split}\label{3.45}
\chi\int_0^\infty \frac{ W^{m-1}|\phi\phi_{x}\psi_{x}|}{U} \leq
CN(t)\int_0^\infty \frac{ W^{m-1}\phi_{x}^{2}}{U}+\chi\varepsilon
N(t)\int_0^\infty \psi_{x}^{2}.\end{split}
\end{equation}
By Young's inequality and  \eqref{3.25}, one has
\[\begin{split}
\chi(1-m)\int_0^\infty  W^{m-1}\phi_{x}\psi^{2}&\leq CN(t)\ii\frac{ W^{m-1}\phi_{x}^{2}}{U}+CN(t)\ii W^{m-1}U\psi_{x}^{2}\\
&\leq CN(t)\ii\frac{W^{m-1}\phi_{x}^{2}}{U}+CN(t)\ii\psi_{x}^{2}.\end{split}\]
By Lemma \ref{hardy} and  \eqref{3.25} again, we get
\begin{equation*}\begin{split}
&\left|\int_0^\infty
W^{m-1}U\psi\sum_{n=2}^{\infty}\frac{(1-m)^{n}\psi^{n}}{n!}\right| \leq C\int_0^\infty  W^{m-1}U|\psi|^{3}\\&
\leq CN(t)\int_0^\infty \left(1+\frac{\theta r}{\beta}x\right)^{-2}\psi^2 \leq CN(t)\ii\psi_x^2,\end{split}
\end{equation*}
and
\begin{equation}\begin{split}\label{3.47}
\left|\int_0^\infty
W^{m-1}\phi_{x}\psi\sum_{n=2}^{\infty}\frac{(1-m)^{n}\psi^{n}}{n!}\right|
&\leq C\int_0^\infty |\phi_{x}||\psi|^{3}\\&
\leq CN(t)\int_0^\infty \frac{ W^{m-1}\phi_{x}^{2}}{U}+CN(t)\int_0^\infty W^{m-1}U\psi^{2}\\
&\leq CN(t)\int_0^\infty \frac{
W^{m-1}\phi_{x}^{2}}{U}+CN(t)\int_0^\infty\psi_x^2.\end{split}
\end{equation}
Now substituting \eqref{3.44}, \eqref{3.45}-\eqref{3.47} into
\eqref{3.39}, we have
\begin{equation*}
\begin{split}&\int_0^\infty \left(\frac{ W^{m-1}\phi^{2}}{U}+\chi \psi^{2}\right)
+(1-CN(t))\int_0^t\int_0^\infty \frac{ W^{m-1}\phi_{x}^{2}}{U}+(\chi\varepsilon-CN(t))\int_0^t \int_0^\infty \psi_{x}^{2}\\
&\leq \int_0^\infty \left(\frac{
W^{m-1}\phi_{0}^{2}}{U}+\chi\psi_{0}^{2}\right).
\end{split}
\end{equation*}
Therefore, \eqref{3.38} holds provided that $N(t)\ll1$.
\end{proof}

\begin{lemma}\label{H1b}
If $N(t)\ll1$, then
\begin{equation}\label{3.49}
\begin{split}\int_0^\infty \left(\frac{ W^{m-1}\phi_x^{2}}{U}+\psi_x^{2}\right)
+\int_0^t \int_0^\infty \left(\frac{
W^{m-1}\phi_{xx}^{2}}{U}+\psi_{xx}^{2}\right) \leq
C(\|\phi_{0}\|^2_{1,w_3}+\|\psi_{0}\|^2_{1}).
\end{split}
\end{equation}

\end{lemma}

\begin{proof}
Multiplying the first equation of \eqref{newfai} by $\frac{
W^{m-1}\phi_{xx}}{U}$ yields
\begin{equation*}
\begin{split}
\frac{d}{dt}\int_0^\infty \frac{ W^{m-1}\phi_x^{2}}{2U}
+\int_0^\infty \frac{ W^{m-1}\phi_{xx}^{2}}{U}&=-\chi\int_0^\infty
W^{m-1}\left(\frac{V\phi_{x}}{U}+\psi_{x}
+\frac{\phi_{x}\psi_{x}}{U}\right)\phi_{xx}
\\&\quad-\ii\phi_t\phi_{x}\left(\frac{ W^{m-1}}{U}\right)_x.
\end{split}
\end{equation*}
By Young's inequality and \eqref{3.25}, we have
$$\chi\left|\frac{ W^{m-1}V\phi_{x}}{U}\phi_{xx}\right|\leq \frac{W^{m-1}\phi_{xx}^{2}}{4U}+\frac{\chi^2 W^{m-1}V^2\phi_{x}^{2}}{U}
\leq\frac{W^{m-1}\phi_{xx}^{2}}{4U}+\frac{C
W^{m-1}\phi_{x}^{2}}{U},$$
$$\chi\left| W^{m-1}\psi_{x}\phi_{xx}\right|\leq\frac{W^{m-1}\phi_{xx}^{2}}{4U}
+\chi^2 W^{m-1}U\psi_{x}^{2}\leq \frac{
W^{m-1}\phi_{xx}^{2}}{4U}+C\psi_{x}^{2},$$
$$\chi \left| W^{m-1}\phi_{x}\psi_{x}\frac{\phi_{xx}}{U}\right|
\leq\frac{N(t) W^{m-1}\phi_{xx}^{2}}{4U}+\frac{N(t)\chi^2
W^{m-1}\phi_{x}^{2}}{U},$$ and
\[\begin{split}-\phi_t\phi_{x}\left(\frac{ W^{m-1}}{U}\right)_x&=-(\phi_{xx}+\chi V\phi_{x}+\chi U\psi_{x}+\chi\phi_{x}\psi_{x})\frac{\phi_{x}}{U}\left(( W^{m-1})_x-\frac{ W^{m-1}U_{x}}{U}\right)\\
&\leq\frac{W^{m-1}\phi_{xx}^{2}}{4U}+\frac{C
W^{m-1}\phi_{x}^{2}}{U}+C\psi_{x}^{2}.\end{split}\] Thus,
\begin{equation}\label{3.51}
\frac{d}{dt}\int_0^\infty \frac{
W^{m-1}\phi_{x}^{2}}{2U}+\frac{1}{4}\left(1-N(t)\right)\int_0^\infty
\frac{ W^{m-1}\phi_{xx}^{2}}{U}\leq C\left(\int_0^\infty \frac{
W^{m-1}\phi_{x}^{2}}{U}+\int_0^\infty \psi_{x}^{2}\right).
\end{equation}

Multiplying the second equation of \eqref{newfai} by $\psi_{xx}$,
and noting
\begin{equation*}
\begin{split}\psi_{t}\psi_{xx}=(\psi_{t}\psi_{x})_{x}-\left(\frac{\psi_{x}^{2}}{2}\right)_{t},
\end{split}
\end{equation*}
we get
\begin{equation}\label{3.52}
\begin{split}\frac{1}{2}\frac{d}{dt}\int_0^\infty \psi_x^{2}+\varepsilon\int_0^\infty \psi_{xx}^{2}&=
\ii(2\varepsilon V\psi_x- W^{m-1}\phi_x)\psi_{xx}\\&\quad-\ii W^{m-1}(1-e^{-(m-1)\psi})(U+\phi_x)\psi_{xx}-\varepsilon\ii\psi_x^{2}\psi_{xx}\\
&\leq\frac{\varepsilon}{2}(1+N(t))\int_0^\infty
\psi_{xx}^{2}+C\ii\psi_x^{2}+C\ii\frac{ W^{m-1}\phi_x^2}{U}.
\end{split}\end{equation}
Now integrating \eqref{3.51} and \eqref{3.52} in $t$, by
\eqref{3.38}, we get \eqref{3.49} provided that $N(t)\ll1$.
\end{proof}

Applying the same argument as that of Lemma \ref{H2}, we  have
the following $H^2$-estimates. For brevity, we omit the details of the proof.
\begin{lemma}\label{small-H2}
If $N(t)\ll1$, then
\begin{equation*}
\begin{split}\int_0^\infty \left(\phi_{xx}^{2}+\psi_{xx}^{2}\right)
+\int_0^t \int_0^\infty
\left(\phi_{xxx}^{2}+\psi_{xxx}^{2}\right)\leq
C(\|\phi_{0xx}\|^2+\|\phi_{0}\|^2_{1,w_3}+\|\psi_{0}\|^2_{2}).
\end{split}
\end{equation*}
\end{lemma}

\bigbreak

\begin{remark}
Proposition \ref{mingti-2} follows from Lemmas \ref{lem-small1}, \ref{H1b} and \ref{small-H2}.
\end{remark}

Before we prove our main results, we present a well-known result for convenience.

\begin{lemma}\label{con}
If $f\in W^{1,1}(0,\infty)$ is a nonnegative function, then $f(t) \to 0$ as $t \to \infty$.
\end{lemma}

\subsection{Proof of main results} Now we are ready to prove Theorem \ref{thm-1} and Theorem \ref{thm2}.

\begin{proof}[Proof of Theorem \ref{thm-1}]
The a priori estimates \eqref{priori} in the case $m\geq1$ and
\eqref{priori-2} in the case $0\leq m<1$ guarantee that $N(t)$ is
small for all $t>0$ if $N(0)$ is small enough. Hence, applying the
standard extension argument, one can obtain the global well-posedness
of system \eqref{newfai} with \eqref{newbound2} and
\eqref{newbound1} in $X(0,\infty)$ if $m\geq1$ and in $Y(0,\infty)$
if $0\leq m<1$. In view of \eqref{decomposition}, system
\eqref{newuv}-\eqref{bound2} has a unique global solution
$(u,v)(x,t)$ satisfying \eqref{regularity} and \eqref{2.18},
respectively. Next we proceed to prove the $L^\infty$ convergence \eqref{asym} and $L^1$ convergence \eqref{L1}. We consider the case $m\geq 1$ first. From the estimates \eqref{priori} and \eqref{priori-2}, we claim that
\begin{equation}\label{ww}
\|\phi_x(\cdot,t)\|+\|\psi_x(\cdot,t)\|\rightarrow0 \text{ as } t\rightarrow +\infty.
\end{equation}
Indeed to prove \eqref{ww}, we just need to verify that $\|\phi_x(\cdot,t)\| \in W^{1,1}(0,\infty)$ and $\|\psi_x(\cdot,t)\| \in W^{1,1}(0,\infty)$ from Lemma \ref{con}. We first prove the former one:  $\|\phi_x(\cdot,t)\| \in W^{1,1}(0,\infty)$. From \eqref{Un} and Lemma \ref{lem-large1}, one has
\begin{equation}\label{con0}
\int_0^\infty \int_0^\infty \phi_x^2 \leq \bar{u} \int_0^\infty \int_0^\infty \frac{\phi_x^2}{U} <\infty.
\end{equation}
Moreover from the results of Proposition \ref{global-existence} along with the Sobolev inequality, we have $\|\psi_x\|_{L^\infty} \leq c_0$ for some positive constant $c_0$. Then using the first equation of \eqref{newfai} and positiveness of $w_1$ and $w_2$ (see \eqref{w}), we can find positive constant $c_i (i=1,2,3,4$) such that  \begin{eqnarray}\label{con1}
\begin{aligned}
\frac{d}{dt} \int_0^\infty \phi_x^2 &=-\int_0^\infty \phi_{xx} \phi_t\\
&=\int_0^\infty \phi_{xx}(\phi_{xx}+\chi V \phi_x +\chi U \psi_x+ \chi \phi_x \psi_x)\\
&\leq c_1 \int_0^\infty \|\phi_x\|_{1,w_1}^2 +c_2\int_0^\infty \|\psi_x\|_{w_2}^2+c_3 \|\psi_x\|_{L^\infty} \int_0^\infty (\phi_x^2+\phi_{xx}^2)\\
& \leq c_4 \int_0^\infty \|\phi_x\|_{1,w_1}^2 +c_2\int_0^\infty \|\psi_x\|_{w_2}^2,
\end{aligned}
\end{eqnarray}
where we have used the uniform boundedness of $U(x)$ and $V(x)$. Then we integrate \eqref{con1} on both sides with respect to $t$ and use \eqref{priori} to get
\begin{eqnarray*}\label{con2}
\begin{aligned}
\int_0^\infty \frac{d}{dt} \int_0^\infty \phi_x^2  \leq c_4 \int_0^\infty \int_0^\infty \|\phi_x\|_{1,w_1}^2 +c_2\int_0^\infty \int_0^\infty \|\psi_x\|_{w_2}^2<\infty
\end{aligned}
\end{eqnarray*}
which together with \eqref{con0} implies that $\|\phi_x(t)\| \in W^{1,1}(0,\infty)$. Then from Lemma \ref{con}, it follows that $\|\phi_x(\cdot,t)\| \text{ as } t\to+\infty$. By similar arguments, we have $\|\psi_x(\cdot,t)\| \text{ as } t\to+\infty$. Therefore the claim \eqref{ww} is proven.

By Cauchy-Schwarz inequality and \eqref{priori}, we find
\[
\begin{split}
\phi_x^2(x,t)=2\left|\int_{0}^x\phi_x\phi_{xx}(y,t)dy\right|
             \leq 2\left(\int_0^\infty\phi_x^2dy\right)^{1/2}\left(\int_0^\infty\phi_{xx}^2dy\right)^{1/2}
             \to 0
\end{split}
\]
as $t\to+\infty$. This implies $\sup\limits_{x\in\R_+}|\phi_x(x,t)|\to0 \text{ as } t\to+\infty$. Similarly, we can show that
\begin{equation*}
\sup\limits_{x\in\R_+}|\psi_x(x,t)|\to0 \text{ as } t\to+\infty.
\end{equation*}
which gives the convergence \eqref{asym}.

Next we prove the $L^1$ convergence. For the case $m\geq 1$, with Lemma \ref{H1}, we find a constant $c_5>0$ depending upon initial value only such that
\begin{equation}\label{L1e1}
\int_0^\infty \frac{\phi_x^2}{U}\leq c_5.
\end{equation}
Next using the fact $W^{1-m} b^{1-m} \geq 1$ (see \eqref{w}), we have from \eqref{newe} and Lemma \ref{lem-large1} that
\begin{eqnarray}\label{L1e2}
\begin{aligned}
\int_0^\infty \bigg(\frac{d}{dt}\int_0^\infty \frac{\phi_x^2}{U}\bigg) &\leq C \bigg(\int_0^\infty \int_0^\infty\frac{\phi_x^2}{U}+ \int_0^\infty \int_0^\infty\ \psi_x^2 \bigg)\\
&\leq C \bigg(\int_0^\infty \int_0^\infty\frac{\phi_x^2}{U}+ b^{1-m} \int_0^\infty \int_0^\infty\ W^{1-m}\psi_x^2 \bigg) \leq c_6
\end{aligned}
\end{eqnarray}
where $c_6>0$ is a constant depending on initial value only.

Then the combination of \eqref{L1e1} and \eqref{L1e2}, along with Lemma \ref{con}, gives
\begin{equation*}
\int_0^\infty \frac{\phi_x^2}{U} \to 0 \ \text{as} \ t \to \infty,
\end{equation*}
which thus with the help of H\"{o}lder inequality yields
$$
\int_0^\infty |\phi_x| \leq \left( \int_0^\infty \frac{\phi_x^2}{U}\right)^{1/2} \left( \int_0^\infty U \right)^{1/2} \ \to 0 \ \text{as} \ t \to \infty
$$
due to the fact that $U\geq 0$ is integrable over $(0,\infty)$. This gives the $L^1$ convergence \eqref{L1}. Finally  analogous arguments show the same result for the case $0<m<1$. Then the proof of Theorem \ref{thm-1} is complete.
\end{proof}

\begin{proof}[Proof of Theorem \ref{thm2}]
Since the transformed system \eqref{newuv} and the original Keller-Segel system \eqref{KS} share the same solution component $u$, it remains only to pass the results from $v$ to $w$ to complete the proof of  Theorem \ref{thm2}. By \eqref{decomposition} and Theorem
\ref{thm-1}, we get the regularity of $w_x/w-W_x/W$. We proceed to prove the results for $w-W$. Set $\xi:=w- W$. By
\eqref{trans} and \eqref{decomposition}, \[ W=b e^{-\int_0^xV(y)dy}
\text{ and } w(x,t)=b e^{-\int_0^xv(y,t)dy}=b
e^{-\int_0^x(\psi_x+V)dy}=e^{-\psi}W.\] Thus, $\xi$ satisfies
\begin{equation}\label{3.53}
\begin{split}
\xi_t-\varepsilon\xi_{xx}=UW^m-uw^m=U
W^m(1-e^{-m\psi})-W^m\phi_xe^{-m\psi},
\end{split}\end{equation}
with initial and boundary conditions
\begin{equation*}
 \xi(x,0)=\xi_{0}(x), \ \xi(0,t)=\xi(+\infty,t)=0.
\end{equation*}
By Taylor expansion, since $\|\psi(\cdot,t)\|_{L^\infty}\leq
N(t)\ll1$, it follows
\[|e^{-m\psi}-1|=|m\psi+\sum_{n=2}^{\infty}\frac{(-m)^{n}\psi^{n}}{n!}|\leq C|\psi|, \text{ and } e^{-m\psi}\leq C.\]
Multiplying \eqref{3.53} by $\xi$ and using Young's inequality, we
have
\begin{equation*}
\begin{split}\frac{1}{2}\frac{d}{dt}\int_0^\infty \xi^{2}+\varepsilon\int_0^\infty \xi_{x}^{2}
&=\ii W^m[U(1-e^{-m\psi})-\phi_xe^{-m\psi}]\xi\\
&\leq h\ii  W^{m-1}U\xi^2+\frac{C}{2h}\ii\left(U W^{m+1}\psi^2+\frac{ W^{m+1}\phi_x^2}{U}\right)\\
&\leq\frac{2h\beta b^{m-1}}{r^2}\ii\xi_{x}^{2}+C\left(\ii
U\psi^2+\ii\frac{\phi_x^2}{U}\right),
\end{split}\end{equation*}
where $h$ is a small constant, and we have used Young's inequality
in the first inequality, and \eqref{3.25}, Lemma \ref{hardy} and $
W^{m+1}\leq b^{m+1}$ in the second inequality. Integrating this
inequality in $t$, taking $h=\frac{\varepsilon r^2}{4\beta
b^{m-1}}=\frac{(\chi+m-1)^2}{16(\chi+m+1)}$, and using Lemmas
\ref{lem-large1} and \ref{lem-small1}, we have
\begin{equation}\label{3.54}
\begin{split}
\int_0^\infty \xi^{2}+\varepsilon\int_0^t\int_0^\infty
\xi_{x}^{2}\leq\int_0^\infty \xi_0^{2}
+C(\|\phi_{0}\|^2_{w_i}+\|\psi_{0}\|^2_{w_j}).
\end{split}\end{equation}
Here $w_i=\frac{1}{U}$, $w_j= W^{1-m}$ if $m\geq1$, and $w_i=\frac{
W^{m-1}}{U}$, $w_j=1$ if $0\leq m<1$.

To estimate the first order derivative of $\xi$, we multiply
\eqref{3.53} by $\xi_{xx}$ to get
\begin{equation*}
\begin{split}\frac{1}{2}\frac{d}{dt}\int_0^\infty \xi_x^{2}+\varepsilon\int_0^\infty \xi_{xx}^{2}
&=\ii W^m[U(1-e^{-m\psi})-\phi_xe^{-m\psi}]\xi_{xx}\\
&\leq \frac{\varepsilon}{2}\ii \xi_{xx}^2+\frac{C}{2\varepsilon}\ii\left(U^2 W^{2m}\psi^2+ W^{2m}\phi_x^2\right)\\
&\leq\frac{\varepsilon}{2}\ii \xi_{xx}^2+C\left(\ii U\psi^2+\ii
\phi_x^2\right),
\end{split}\end{equation*}
where we have used $ W^{2m}U^2\leq\frac{b^{2m}\theta^4}{4\beta^2}$
and $ W^{2m}\leq b^{2m}$. Thus, integrating this inequality in $t$
and using Lemmas \ref{lem-large1} and \ref{lem-small1}, we have
\begin{equation}\label{3.55}
\begin{split}
\int_0^\infty \xi_x^{2}+\varepsilon\int_0^t\int_0^\infty
\xi_{xx}^{2}\leq\int_0^\infty \xi_{0x}^{2}
+C(\|\phi_{0}\|^2_{w_i}+\|\psi_{0}\|^2_{w_j}).
\end{split}\end{equation}
By \eqref{3.54} and \eqref{3.55}, one can see that
$\|\xi_x\|\rightarrow0$ as $t\rightarrow\infty$. Then by
Cauchy-Schwarz inequality, we get
\[\xi^2(x,t)=2\int_0^x\xi\xi_x\leq2\left(\ii\xi^2\right)^{\frac{1}{2}}\left(\ii\xi_x^2\right)^{\frac{1}{2}}\leq C\|\xi_x\|.\]
Hence,
\[\sup\limits_{x\in\R_+}|\xi(x,t)|\leq C\|\xi_x\|\rightarrow0 \text{ as }t\rightarrow\infty\]
which completes the proof of Theorem \ref{thm2}.
\end{proof}

\section*{Acknowledgements}
JAC was partially supported by the EPSRC grant EP/P031587/1.
JL's work is partially supported by the National Science
Foundation of China (No. 11571066). He is also grateful for the
hospitality of Hong Kong Polytechnic University where part of this
work was done. The research of ZW is supported by the Hong Kong
RGC GRF grant No. PolyU 153032/15P (Project ID P0005368).

\setlength{\bibsep}{0.5ex}
\bibliography{rf}

\bibliographystyle{plain}

\end{document}